\documentclass[11pt]{amsart}

\usepackage{amsaddr}
\usepackage{caption}
\usepackage{enumerate}
\usepackage{float}
\usepackage{graphicx}
\usepackage{geometry} 
\usepackage{setspace}
\usepackage{subcaption}
\usepackage{tabu}

\newcommand{\vu}{u}
\newcommand{\vvu}{\vvct{u}}

\newcommand{\vv}{v}

\newcommand{\vw}{w}
\newcommand{\vvw}{\vvct{w}}

\newcommand{\vn}{\hat{n}}

\newcommand{\vf}{f}

\newcommand{\vo}{0}
\newcommand{\vzero}{0}

\newcommand{\tsigma}{\boldsymbol{\sigma}}

\newcommand{\tzero}{\mathbf{0}}

\newcommand{\tJ}{\mathbf{J}}

\newcommand{\Ch}{C}

\newcommand{\im}{\mathbf{i}}

\newcommand{\av}[1]{ \{\!\!\{#1\}\!\!\} }
\newcommand{\ju}[1]{ [\![#1]\!] }

\newcommand{\vvct}[1]{\underline{#1}}
\newcommand{\ten}[1]{\mathbf{#1}}

\newtheorem{thm}{Theorem}[section]
\newtheorem{lem}[thm]{Lemma}
\newtheorem{rem}[thm]{Remark}
\newtheorem{exm}[thm]{Example}

\title[Penalty Term and Time Step Estimates]{Sharp Penalty Term and Time Step Bounds for the Interior Penalty Discontinuous Galerkin Method for Linear Hyperbolic Problems*}
\author{S. Geevers$^1$, J.J.W. van der Vegt$^1$}
\address{1. Department of Applied Mathematics, University of Twente, Enschede, the Netherlands (e-mail: s.geevers@utwente.nl, j.j.w.vandervegt@utwente.nl)}
\thanks{*This work was funded by the Shell Global Solutions International B.V. under contract no.~PT45999.}

\begin{document}

\maketitle

\begin{abstract}
  We present \emph{sharp} and \emph{sufficient} bounds for the interior penalty term and time step size to ensure stability of the symmetric interior penalty discontinuous Galerkin (SIPDG) method combined with an explicit time-stepping scheme. These conditions hold for generic meshes, including unstructured nonconforming heterogeneous meshes of mixed element types, and apply to a large class of linear hyperbolic problems, including the acoustic wave equation, the (an)isotropic elastic wave equations, and Maxwell's equations. The penalty term bounds are computed elementwise, while bounds for the time step size are computed at weighted submeshes requiring only a small number of elements and faces. Numerical results illustrate the sharpness of these bounds.
\end{abstract}

\section{Introduction}
Realistic wave problems often involve large three-dimensional domains, with complex boundary layers, sharp material interfaces and detailed internal structures. Solving such problems therefore requires a numerical method that is efficient in terms of both memory usage and computation time and is flexible enough to deal with interfaces and internal structures without leading to an unnecessary overhead. 

A standard finite difference scheme therefore falls short, since it cannot efficiently deal with complex material interfaces, and since small internal structures impose global restrictions on the grid resolution. Finite element methods overcome these problems, since they can be applied to unstructured meshes. However, the finite element method, combined with an explicit time-stepping scheme, requires solving mass matrix-vector equations during every time step. This significantly increases the computational time when the mass matrix is not (block)-diagonal. To obtain diagonal mass matrices without losing the order of accuracy, several mass-lumping techniques have been developed; see, for example, \cite{komatitsch99, chin99, cohen01, mulder13}. However, for higher order elements, these techniques require additional quadrature points and degrees of freedom to maintain the optimal order of accuracy. 

An alternative is the discontinuous Galerkin (DG) finite element method. This method is similar to the conforming finite element method but allows its approximation functions to be discontinuous at the element boundaries, which naturally results in a block-diagonal mass matrix. Additional advantages of this method are that it also supports meshes with hanging nodes, and that the extension to arbitrary higher order polynomial basis functions is straightforward and can be adapted elementwise. The downside of the DG method, however, is that the discontinuous basis functions can result in significantly more degrees of freedom. 

Still, because of its advantages, numerous DG schemes have already been developed and analyzed for linear wave problems, including the symmetric interior penalty discontinuous Galerkin (SIPDG) method; see, for example, \cite{grote06, grote08, antonietti12}. The advantage of the SIPDG method is that it is based on the second order formulation of the problem, while schemes based on a first order formulation require solving additional variables leading to more memory usage. The SIPDG and several alternatives have also been compared and analyzed in \cite{arnold02}, from which it follows that the SIPDG method is one of the most attractive options because of its consistency, stability, and optimal convergence rate. 

However, to efficiently apply the SIPDG method with an explicit time-stepping scheme, the interior penalty term needs to be sufficiently large and the time step size needs to be sufficiently small. If the penalty term is set too small or the time step size too large, the SIPDG scheme will become unstable. On the other hand, increasing the penalty term will lead to a more severe time step restriction, and a smaller time step size results in a longer computation time. For this reason, there have been multiple studies on finding suitable choices for these parameters.

In \cite{shahbazi05, epshteyn07}, for example, sufficient conditions have been derived for the penalty term, for triangular and tetrahedral meshes, while the results of \cite{epshteyn07} have been sharpened in \cite{mulder14} for tetrahedral meshes. However, the numerical results in Section \ref{sec:numResults} illustrate that these estimates are still not always very sharp. In \cite{mulder14} they also studied the effects of the penalty term, element shape, and polynomial order on the CFL condition for tetrahedral elements, although these results may not give sufficient conditions for nonuniform grids. Penalty term conditions for regular square and cubic meshes have been studied in \cite{ainsworth06, deBasabe10, agut13}, where \cite{deBasabe10, agut13} also studied the CFL conditions for these meshes. However, the analysis in these studies only holds for uniform homogeneous meshes. For generic heterogeneous meshes, sharp parameter conditions have remained an open problem. 

In this paper we solve these problems by deriving sufficient conditions for both the penalty term and time step size, which lead to sharp estimates, and which hold for generic meshes, including unstructured nonconforming heterogeneous meshes of mixed element types with different types of boundary conditions. These conditions also apply to a large class of linear wave problems, including the acoustic wave equation, Maxwell's equations, and (an)isotropic elastic wave equations. We compare our estimates to some of the existing ones and illustrate the sharpness of our parameter estimates with several numerical tests. 

The paper is constructed as follows: in Section \ref{sec:tenNotation} we introduce some tensor notation, such that we can present the general linear hyperbolic model in Section \ref{sec:genHypModel}, and present the symmetric interior penalty discontinuous Galerkin method in Section \ref{sec:SIPDG}. After this, we derive sufficient conditions for the penalty parameter in Section \ref{sec:estPen} and sufficient conditions for the time step size in Section \ref{sec:timeStepEst}. Finally, we compare and test the sharpness of our estimates in Section \ref{sec:numResults} and give a summary in Section \ref{sec:conclusion}.

\section{Some tensor notation}
\label{sec:tenNotation}
Before we present the linear hyperbolic model, it is useful to define some tensor notation. Let $M$ and $N$ be two nonnegative integers. Also, let $A\in\mathbb{R}^{n_1\times\cdots\times n_N}$ be a tensor of order $N$ and $B\in\mathbb{R}^{m_1\times\cdots\times m_M}$ be a tensor of order $M$. We define the tensor product $AB\in\mathbb{R}^{n_1\times\cdots\times n_N\times m_1\times\cdots\times m_M}$, which is of order $N+M$, as follows:
\begin{align*}
[AB]_{i_1,\dots,i_N,j_1,\dots,j_M} :=& A_{i_1,\dots,i_N}B_{j_1,\dots,j_M},
\end{align*}
for $(i_1,\dots,i_N,j_1,\dots,j_M)\in(\mathbb{N}_{n_1},\dots,\mathbb{N}_{n_N},\mathbb{N}_{m_1},\dots,\mathbb{N}_{m_M})$, where $\mathbb{N}_n := \{1,\dots,n\}$. Now let $A\in\mathbb{R}^{n_1\times\cdots\times n_N\times p}$ be a tensor of order $N+1$ and $B\in\mathbb{R}^{p\times m_1\times\cdots\times m_M}$ be a tensor of order $M+1$. We define the dot product $A\cdot B\in\mathbb{R}^{n_1\times\cdots\times n_N\times m_1\times\cdots\times m_M}$, which is of order $N+M$, as follows:
\begin{align*}
[A\cdot  B]_{i_1,\dots,i_N,j_1,\dots,j_M} :=& \sum_{k=1}^p A_{i_1,\dots,i_N,k}B_{k,j_1,\dots,j_M},
\end{align*}
for $(i_1,\dots,i_N,j_1,\dots,j_M)\in(\mathbb{N}_{n_1},\dots,\mathbb{N}_{n_N},\mathbb{N}_{m_1},\dots,\mathbb{N}_{m_M})$. For the double dot product, let $A\in\mathbb{R}^{n_1\times\cdots\times n_N\times p_1\times p_2}$ be a tensor of order $N+2$ and $B\in\mathbb{R}^{p_2\times p_1\times m_1\times\cdots\times m_M}$ be a tensor of order $M+2$. We define $A:B\in\mathbb{R}^{n_1\times\cdots\times n_N\times m_1\times\cdots\times m_M}$, which is of order $N+M$, as follows:
\begin{align*}
[A:  B]_{i_1,\dots,i_N,j_1,\dots,j_M} :=& \sum_{k_1=1}^{p_1}\sum_{k_2=1}^{p_2} A_{i_1,\dots,i_N,k_1,k_2}B_{k_2,k_1,j_1,\dots,j_M},
\end{align*}
for $(i_1,\dots,i_N,j_1,\dots,j_M)\in(\mathbb{N}_{n_1},\dots,\mathbb{N}_{n_N},\mathbb{N}_{m_1},\dots,\mathbb{N}_{m_M})$. Now let $A\in\mathbb{R}^{n_1\times\cdots\times n_N}$ be a tensor of order $N$ again. We define the transpose $A^t\in\mathbb{R}^{n_N\times\cdots\times n_1}$ as follows:
\begin{align*}
A^t_{i_N,\dots,i_1} := A_{i_1,\dots,i_N},
\end{align*}
for $(i_N,\dots,i_1)\in(\mathbb{N}_{n_N},\dots,\mathbb{N}_{n_1})$. A tensor $A$ is called symmetric if $A=A^t$ and we define $\mathbb{R}^{n_1\times\cdots\times n_N}_{sym}$ to be the set of symmetric tensors in $\mathbb{R}^{n_1\times\cdots\times n_N}$. Now let $B\in\mathbb{R}^{n_1\times\cdots\times n_N}$ be a tensor of the same size as $A$. We define the inner product as follows:
\begin{align*}
(A,B) := \sum_{i_1=1}^{n_1}\cdots\sum_{i_N=1}^{n_N}A_{i_1,\dots,i_N}B_{i_1,\dots,i_N}.
\end{align*}
The corresponding norm of a tensor is given by
\begin{align*}
\|A \|^2 := (A,A).
\end{align*}
In the next section we will present the general linear hyperbolic problem, which we will solve using the SIPDG method.

\section{A general linear hyperbolic model}
\label{sec:genHypModel}
Let $\Omega\subset\mathbb{R}^d$ be a $d$-dimensional open domain with a Lipschitz boundary $\partial\Omega$, and let $(0,T)$ be the time domain. Also, let $\{\Gamma_d,\Gamma_n\}$ be a partition of $\partial\Omega$, corresponding to Dirichlet and von Neumann boundary conditions, respectively. We define the following linear hyperbolic problem:
\begin{subequations}
\begin{align}
\rho\partial_t^2\vu &= \nabla\cdot C:\nabla \vu + \vf &&\text{in }\Omega\times(0,T), \\
\vn\cdot C:\vn\vu &=\vo &&\text{on }\Gamma_d\times(0,T), \\
\vn\cdot C:\nabla\vu &= \vo &&\text{on }\Gamma_n\times(0,T), \\
\vu|_{t=0} &= \vu_0 &&\text{in }\Omega, \\
\partial_t\vu|_{t=0} &= \vv_0 && \text{in }\Omega,
\end{align}
\label{eq:model}%
\end{subequations}
where $\vu:\Omega\times(0,T)\rightarrow\mathbb{R}^m$ is a vector of $m$ variables that are to be solved, $\nabla$ is the gradient operator, $\rho:\Omega\rightarrow\mathbb{R}^+$ is a positive scalar field, $C:\Omega\rightarrow\mathbb{R}^{d\times m\times m\times d}_{sym}$ a fourth order tensor field, $\vf:\Omega\times(0,T)\rightarrow\mathbb{R}^{m}$ the external volume force, and $\vn:\partial\Omega\rightarrow\mathbb{R}^d$ the outward pointing normal unit vector. 

We make some assumptions on the material parameters. First, we assume that there exist positive constants $\rho_{min}, \rho_{max}$ such that 
\begin{align}
0<\rho_{min}&\leq\rho\leq\rho_{max} &&\text{in }\Omega.
\label{eq:rhoBound}
\end{align}
We also assume that the tensor field is symmetric, $C=C^t$, and that there exist linear subspaces $\Sigma_0, \Sigma_1\subset\mathbb{R}^{d\times m}$, with $\Sigma_0+\Sigma_1=\mathbb{R}^{d\times m}$ and $\Sigma_0 \perp\Sigma_1$, and constants $c_{min},c_{max}$ such that $C$ is nonnegative and bounded in the following sense:
\begin{subequations}
\begin{align}
0<c_{min} \leq \frac{\tsigma^t:C:\tsigma }{\| \tsigma \|^2} &\leq c_{max} &&\text{in }\Omega,\text{ for all } \tsigma\in\Sigma_1/\{\ten{0}\}, \\
C:\tsigma &=\ten{0} &&\text{in }\Omega,\text{ for all } \tsigma\in\Sigma_0.
\end{align}
\label{eq:cBound}%
\end{subequations}
By $\Sigma_0+\Sigma_1=\mathbb{R}^{d\times m}$ we mean that any $\tsigma\in\mathbb{R}^{d\times m}$ can be written as $\tsigma=\tsigma_0+\tsigma_1$ for some $\tsigma_0\in\Sigma_0, \tsigma_1\in\Sigma_1$, and by $\Sigma_0\perp\Sigma_1$ we mean that $(\tsigma_0,\tsigma_1)=0$ for all $\tsigma_0\in\Sigma_0, \tsigma_1\in\Sigma_1$.

By choosing the correct tensor and scalar field we can obtain a number of hyperbolic problems, including Maxwell's equations, the acoustic wave equation and the (an)isotropic elastic wave equations. We illustrate this in the following examples, where we define the tensor and vector fields using Cartesian coordinates.
\begin{exm} 
Consider the acoustic wave equation written in the following dimensionless form:
\begin{align*}
\partial_t^2 p = \nabla\cdot c^2\nabla p,
\end{align*}
where $p:\Omega\times(0,T)\rightarrow\mathbb{R}$ is the pressure field and $c:\Omega\rightarrow\mathbb{R}^+$ the acoustic wave velocity. We can write these equations in the form of (\ref{eq:model}) by setting $m=1$, $\vu= p$, $\vf=\vo$, $\rho=1$, and
\begin{align*}
C_{i,1,1,j} := \delta_{ij}c^2
\end{align*}
for $i,j=1,\dots,d$, where $\delta$ is the Kronecker delta function.
\label{exm:acousticWave}
\end{exm} 

\begin{exm}
Consider Maxwell's equations in a domain with zero conductivity written in the following dimensionless form:
\begin{align*}
\epsilon\partial_t^2E=-\nabla\times(\mu^{-1}\nabla\times E) + \partial_tj,
\end{align*}
where $E:\Omega\times(0,T)\rightarrow\mathbb{R}^3$ is the electric field, $\epsilon:\Omega\rightarrow\mathbb{R}^+$ the relative electric permittivity, $\mu:\Omega\rightarrow\mathbb{R}^+$ the relative magnetic permeability, and $j:\Omega\times(0,T)\rightarrow\mathbb{R}^3$ the applied current density. We can write these equations in the form of (\ref{eq:model}) by setting $d=3$, $m=3$, $\vu= E$, $\vf=\partial_tj$, $\rho=\epsilon$, and
\begin{align*}
C_{i_D,i_M,j_M,j_D} := \mu^{-1}\big(\delta_{i_D,j_D}\delta_{i_M,j_M} - \delta_{i_D,j_M}\delta_{i_M,j_D}\big)
\end{align*}
for $i_D,i_M,j_D,j_M=1,2,3$, where $\delta$ is the Kronecker delta function.
\label{exm:Maxwell}
\end{exm}

\begin{exm} 
Consider the linear anisotropic elastic wave equations. They can immediately be written in the form of (\ref{eq:model}) with $\vu:\Omega\times (0,T) \rightarrow \mathbb{R}^d$ being the displacement vector and $C:\Omega\rightarrow\mathbb{R}^{d\times d\times d\times d}_{sym}$ being the elasticity tensor, which has the additional symmetries
\begin{align*}
C_{i_D,i_M,j_M,j_D} = C_{i_M,i_D,j_M,j_D} = C_{i_D,i_M,j_D,j_M}
\end{align*}
for $i_D,i_M,j_D,j_M=1,\dots,d$. For the special case of isotropic elasticity, we can write
\begin{align*}
C_{i_D,i_M,j_M,j_D} = \lambda\delta_{i_D,i_M}\delta_{j_D,j_M} + \mu(\delta_{i_D,j_D}\delta_{i_M,j_M} + \delta_{i_D,j_M}\delta_{i_M,j_D})
\end{align*}
for $i_D,i_M,j_D,j_M=1,\dots,d$, where $\delta$ is the Kronecker delta function and $\lambda, \mu$ are the Lam\'e parameters.
\label{exm:anisotropicWave}
\end{exm}

In the next section we present the DG method that we use to solve these linear hyperbolic problems.

\section{A discontinuous Galerkin method}
\label{sec:SIPDG}
To explain the DG method, we first present the weak formulation of (\ref{eq:model}). After that, we introduce the tesselation of the domain, the discrete function spaces, and the trace operators. We then present the symmetric interior penalty discontinuous Galerkin (SIPDG) method and derive some of its properties.

\subsection{The weak formulation}
Define the following function space:
\begin{align*}
U& :=\big\{ \vu\in L^2(\Omega)^m \;\big|\; C:\nabla\vu\in L^2(\Omega)^{d\times m}, \;\vn\cdot C:\vn\vu = 0 \text{ on }\Gamma_d \big\}. 
\end{align*}
Assume that $\vu_0\in U$, $\vv_0\in L^2(\Omega)^m$, and $\vf\in L^2\big(0,T;L^{2}(\Omega)^m\big)$. The weak formulation of (\ref{eq:model}) is finding $\vu\in L^2\big(0,T;U\big)$, with $\partial_t\vu\in L^2\big(0,T; L^2(\Omega)^m\big)$ and $\partial_t(\rho\partial_t\vu)\in L^2\big(0,T;U^{-1}\big)$, such that $\vu|_{t=0}=\vu_0$, $\partial_t \vu|_{t=0} = \vv_0$, and
\begin{align}
\langle \partial_t(\rho\partial_t\vu), \vw\rangle + a(\vu,\vw) &= ( \vf,\vw ), &&\text{a.e. }t\in(0,T),\text{ for all }\vw\in U.
\label{eq:weakForm}
\end{align}
Here $(\cdot,\cdot)$ denotes the inner product of $L^2(\Omega)^m$, $\langle \cdot,\cdot \rangle$ denotes the pairing between $U^{-1}$ and $U$, and $a(\cdot,\cdot):U\times U\rightarrow\mathbb{R}$ is the semielliptic bilinear operator given by
\begin{align*}
a(\vu,\vw) := \int_{\Omega} (\nabla\vu)^t : C : \nabla\vw \; dx.
\end{align*}
Using (\ref{eq:cBound}) it can be shown that $U$ is a separable Hilbert space, and from (\ref{eq:rhoBound}) it follows that the norm $\|\vu\|_{\rho}^2:=(\rho\vu,\vu)$ is equivalent to the standard $L^2(\Omega)^m$ inner product. Using these properties, it can be proven, in a way analogous to the proof of \cite[Chapter 3, Theorem 8.1]{lions72} that (\ref{eq:weakForm}) is well-posed and has a unique solution.

\subsection{Tesselation, discrete function space, and trace operators} 
Let  $\mathcal{T}_h$ be a set of nonoverlapping open domains in $\mathbb{R}^d$, referred to as elements, such that every element $e\in\mathcal{T}_h$ fits inside a $d$-dimensional sphere of radius $h$, and such that $\overline{\Omega} := \bigcup_{e\in\mathcal{T}_h} \overline{e}$, where $\overline{e}$ and $\overline{\Omega}$ are the closures of $e$ and $\Omega$, respectively. We call $\mathcal{T}_h$ the tesselation of $\Omega$. Using the tesselation we define the set of faces $\mathcal{F}_h :=\mathcal{F}_{h,in}\cup \mathcal{F}_{h,b}$ and the union of all faces $\Gamma_{h} := \bigcup_{f\in\mathcal{F}_{h}} f$, where $\mathcal{F}_{h,in} := \big\{ \partial e\cap \partial e' \big\}_{e,e'\in\mathcal{T}_h}$ is the set of all internal faces, $\mathcal{F}_{h,b} := \big\{ \partial e\cap \partial\Omega \big\}_{e\in\mathcal{T}_h}$ is the set of all boundary faces, and $\partial e$ denotes the element boundary. Furthermore, we let $\{ \mathcal{F}_{h,d}, \mathcal{F}_{h,n}\}$ be the partition of $\mathcal{F}_{h,b}$ corresponding to the Dirichlet and Neumann boundary conditions, such that $\bigcup_{f\in\mathcal{F}_{h,d}} f = \Gamma_d$ and $\bigcup_{f\in\mathcal{F}_{h,n}} f = \Gamma_n$.

We use these sets of elements and faces to construct the discrete function space. To do this, let $e\in\mathcal{T}_h$ be an element with $\tilde e$ the corresponding reference element, which depends only on the shape of $e$.  For every reference element $\tilde e$ we define a discrete function space $\tilde U_e:\tilde e\rightarrow\mathbb{R}^m$, such that $\tilde U_e = [\mathcal{P}^{K_e}(\tilde e)]^m$, where $\mathcal{P}^{K_e}(\tilde e)$ is a finite set of polynomial functions on $\tilde e$, which depends only on the degree $K_e$ and $\tilde e$. Now let $\phi_e:\tilde e\rightarrow e$ be an invertible polynomial mapping, such that $|\tJ_e| := |\mathrm{det}(\nabla\phi_e)|\geq|\tJ_e|_{min}> 0$ in $e$, for some positive constant $|\tJ_e|_{min}$. Using this mapping and the reference function space $\tilde U_e$, we can construct the function space on the physical element $U_e$ as follows:
\begin{align*}
U_e := \big\{ u:e\rightarrow\mathbb{R}^m \;|\; u=\tilde u \circ\phi_e^{-1}, \text{ for some }\tilde u\in \tilde U_e \big\}.
\end{align*}
This can then be used to construct the discrete finite element space:
\begin{align*}
U_h := \big\{ u\in [L^2(\Omega)]^m \;\big|\; u|_e \in U_e, e\in\mathcal{T}_h \big\}.
\end{align*}
The functions in the finite element space are continuous within every element, but can be discontinuous at the faces between two elements. To construct a discrete version of the bilinear form $a$, which can deal with these discontinuities, we introduce trace operators $\av{\cdot}$ and $\ju{\cdot}$ given below:
\begin{align*}
\av{\phi} \big|_f &:= \frac{1}{|\mathcal{T}_f|}\sum_{e\in\mathcal{T}_f} \phi|_{\partial e\cap f}, && f\in\mathcal{F}_h, \\
\ju{\vu} \big|_f &:= \sum_{e\in\mathcal{T}_f} (\vn \vu)|_{\partial e\cap f}, && f\in\mathcal{F}_h,
\end{align*}
where $\mathcal{T}_f$ is the set of elements adjacent to $f$, $|\mathcal{T}_f|$ is the number of elements adjacent to $f$, and $\vn|_{\partial e}$ is the outward pointing normal vector of element $e$. The first trace operator is the average of traces, while the second operator is known as the jump operator. Using the first trace operator $\av{\cdot}$, we can construct the numerical flux operator $(\cdot)^*:U_h\rightarrow [L^2(\Gamma_h)]^m$, which assigns a unique value for $\vu\in U_h$ at the faces, as follows:
\begin{align}
\vu^*|_f &= \begin{cases}
\av{\vu} & f\in\mathcal{F}_{h,in}\cap\mathcal{F}_{h,n}, \\
\vzero  & f\in\mathcal{F}_{d}. \end{cases}
\label{eq:defFlux}
\end{align}

In order to ensure that the discrete bilinear form $a_h$ remains semielliptic, we also introduce penalty terms $\eta_e\in\mathbb{R}^+$ for every element $e\in\mathcal{T}_h$, and a penalty scaling function $\nu_h\in\bigotimes_{e\in\mathcal{T}_h} L^{\infty}(\partial e)$, with $\nu_h> 0$, which means $\nu|_{\partial e}:\partial e\rightarrow\mathbb{R}^+$ for all $e\in\mathcal{T}_h$. The penalty terms $\eta_e$ are positive dimensionless constants for which lower bounds will be derived in Section \ref{sec:estPen}. The function $\nu_h$ scales with order $h^{-1}$ and is chosen as follows:
\begin{align*}
\nu_h|_{\partial e\cap f} &:= \left( \frac{|\tJ_f|}{|\tJ_e|} \right)|_{\partial e\cap f}, && e\in\mathcal{T}_h, f\in\mathcal{F}_e,
\end{align*}
where $\mathcal{F}_e$ denotes the faces adjacent to element $e$, $|\tJ_e|:=|\mathrm{det}(\nabla\phi_e)|$ is the reference-to-physical element scale, and $|\tJ_f|$ is the reference-to-physical face scale. The face scale satisfies $|\tJ_f|=1$ in 1D, $|\tJ_f|=| \partial_1\phi_f|$ in 2D, and $|\tJ_f|=| \partial_1\phi_f \times\partial_2\phi_f |$ in 3D, where $\phi_f:\tilde f\rightarrow f$ is the reference-to-physical face mapping, and $\partial_i\phi_f$ is the derivative of $\phi_f$ in reference coordinate $i$, assuming Cartesian reference coordinates. In our numerical tests we use this scaling function, although the stability analysis in this paper holds for arbitrary positive functions $\nu_h$.

Finally, to ensure that the discrete version of $a$ is well defined we also make the following additional assumptions on the material parameters $\rho$ and $C$:
\begin{align*}
\rho|_e &\in W^{1,\infty}(e), \quad C|_e \in W^{1,\infty}(e)^{d\times m\times m\times d}_{sym}, &&e\in\mathcal{T}_h,
\end{align*}
where $W^{1,\infty}$ denotes the Sobolev space of differentiable functions with uniformly bounded weak derivatives. These assumptions together with the trace inequality imply that the element traces of $C$ and $\rho$ are well defined and bounded.

We have now introduced the function spaces, operators, and parameter assumptions needed to present the DG method in the next subsection.

\subsection{The symmetric interior penalty discontinuous Galerkin method}
We present a DG method, which is known as the symmetric interior penalty discontinuous Galkerkin (SIPDG) method. The SIPDG method is solving $\vu:[0,T]\rightarrow U_h$ such that
\begin{subequations}
\begin{align}
(\rho\partial_t^2\vu,\vw) + a_h(\vu,\vw) =& (\vf,\vw), &&\vw\in U_h, t\in[0,T], \label{eq:DGmethod0} \\
(\rho\vu|_{t=0},\vw) =& (\rho\vu_{0},\vw) &&\vw\in U_h, \\
(\rho\partial_t\vu|_{t=0},\vw)=& (\rho\vv_{0},\vw) &&\vw\in U_h,
\end{align}
\label{eq:DGmethod}%
\end{subequations}
where $a_h:U_h\times U_h\rightarrow\mathbb{R}$ is the discrete version of the elliptic operator, given by
\begin{align}
a_h(\vu,\vw) &:=  a_{h}^{(C)}(\vu,\vw) + a_{h}^{(DG)}(\vu,\vw) + a_{h}^{(DG)}(\vw,\vu) + a_{h}^{(IP)}(\vu,\vw) 
\label{eq:a_hDef1},
\end{align}
with
\begin{align*}
a_{h}^{(C)}(\vu,\vw) &:= \sum_{e\in\mathcal{T}_h} a_e^{(C)}(\vu,\vw):= \sum_{e\in\mathcal{T}_h} \int_{e} (\nabla\vu)^t:\Ch:\nabla\vw \;dx, \\
a_{h}^{(DG)}(\vu,\vw) &:=  \sum_{e\in\mathcal{T}_h} a_{\partial e}^{(DG)}(\vu,\vw) := \sum_{e\in\mathcal{T}_h} \int_{\partial e} (\vu^*-\vu)\vn:\Ch:\nabla\vw \;ds, \\
a_{h}^{(IP)}(\vu,\vw) &:= \sum_{e\in\mathcal{T}_h} \eta_e a_{\partial e}^{(IP)}(\vu,\vw) := \sum_{e\in\mathcal{T}_h} \eta_e \int_{\partial e}(\vu^*-\vu)\vn:\nu_h \Ch:\vn(\vw^*-\vw) \; ds,
\end{align*}
for all $\vu,\vw \in U_h$. The bilinear form $a^{(C)}_h$ is the same as the original elliptic operator $a$ and is the part that remains when both input functions are continuous. The bilinear form $a^{(DG)}_h$ can be interpreted as the additional part that results from partial integration of the elliptic operator $a$ when the first input function is discontinuous. Finally, the bilinear form $a^{(IP)}$ is the part that contains the interior penalty terms needed to ensure stability of the scheme.  

Using the definition of the numerical flux in (\ref{eq:defFlux}), we can rewrite the bilinear forms $a_{h}^{(DG)}$ and $a_{h}^{(IP)}$, as follows:
\begin{subequations}
\begin{align}
a_{h}^{(DG)}(\vu,\vw) &= \sum_{f\in\mathcal{F}_{h,in}\cup\mathcal{F}_{h,d}} - \int_f \ju{\vu}^t:\av{\Ch : \nabla\vw} \;ds, \\
a_{h}^{(IP)}(\vu,\vw) &= \sum_{f\in\mathcal{F}_{h,in}\cup\mathcal{F}_{h,d}} \epsilon_f\int_{f} \ju{\vu}^t:\av{\eta_h\nu_h \Ch}:\ju{\vw} \;ds, \label{eq:aIPdef2}
\end{align} %
\end{subequations}
for all $\vu,\vw \in U_h$. Here $\epsilon_f:=1/2$ for internal faces and $\epsilon_f:=1$ for faces in $\mathcal{F}_{h,d}$, and $\eta_h\in\bigotimes_{e\in\mathcal{T}_h} L^{\infty}(\partial e)$ is defined by $\eta_h|_{\partial e}:= \eta_e$ for all $e\in\mathcal{T}_h$. This scheme conforms with existing SIPDG schemes, except for a possible deviation in the interior penalty part $a_{h}^{(IP)}$. For example, for the acoustic wave equation given in Example \ref{exm:acousticWave}, this scheme is equivalent to the one in \cite{grote06} when choosing their penalty term $\mathbf{a}$ as $\mathbf{a}|_f=\epsilon_f\av{\eta\nu_h c}|_f$ for all $f\in\mathcal{F}_h$. 

Since the bilinear form is symmetric, $a_h(u,w) = a_h(w,u)$ for all $\vu,\vw\in U_h$, we can obtain, by substituting $\vw=\partial_t\vu$ into (\ref{eq:DGmethod0}), the following energy equation:
\begin{align*}
\partial_tE_h &= (\vf,\partial_t\vu), &&t\in[0,T],
\end{align*}
where $E_h:=\frac12\big\| \rho^{1/2}\vw \big\|_{0}^2 + \frac12 a_h(\vu,\vu)$ is the discrete energy, with $\|\cdot\|_{0}$ the $[L^2(\Omega)]^m$ norm. In the absence of an external force $\vf$ this implies that the discrete energy is conserved. 

However, for this energy to be well defined, in the sense that it is always nonnegative, the discrete bilinear form needs to remain semielliptic: $a_h(\vu,\vu)\geq 0$ for all $\vu\in U_h$. This then implies that any nonzero discrete eigenmode cannot grow unbounded in the absence of an external force. In case $\vu$ is a zero discrete eigenmode, we can substitute $\vw=\vu$ into (\ref{eq:DGmethod0}) to obtain $\partial_t^2 \|\rho^{1/2}\vu \|^2_0 = 2(\vf,\vu)$. This implies that, in the absence of an external force, zero eigenmodes grow at most linearly in time. This behavior can correspond to physical rigid motions, or, when there is a discrepancy between the physical and discrete zero modes, a linear drift of a spurious mode. For acoustic and elastic waves these spurious modes are absent, while for electromagnetic waves, these modes have been analyzed in, for example, \cite{buffa06}. However, even if there are spurious modes, we will not consider their drift as numerical instability, since the numerical error is expected to grow linearly in time anyway due to dispersion errors.

In the next section we will find sufficient lower bounds for the penalty term to make sure $a_h$ is semielliptic. In particular, we will show there that $a_h$ satisfies a coercivity condition that is commonly used to show optimal convergence in the energy-norm.

\section{Sufficient penalty term estimates}
\label{sec:estPen}
In this section we derive a sufficient lower bound for the penalty term and a positive constant $c_{coer}>0$, where $c_{coer}$ is independent of the mesh $\mathcal{T}_h$, such that
\begin{align}
a_{h}(\vu,\vu) \geq c_{coer} \big| \vu \big|^2_{1,h}, &&\vu\in U_h,
\label{eq:coerCond2}
\end{align}
where $|\cdot|_{1,h}$ is the discrete seminorm defined by $\big| \vu \big|^2_{1,h} := \sum_{e\in\mathcal{T}_h} \big| \vu \big|^2_{1,e}$, with
\begin{align*}
\big| \vu \big|^2_{1,e} := \int_e \|\Ch^{1/2}:\nabla\vu\|^2 \;dx \;+ \; \eta_e \int_{\partial e}  \big\| \nu_h^{1/2}\Ch^{1/2}:\vn(\vu^*-\vu) \big\|^2 \;ds.
\end{align*}
Here $\Ch^{1/2}\in\bigotimes_{e\in\mathcal{T}_h} W^{1,\infty}(e)^{d\times m\times m\times d}_{sym}$ is a tensor field such that $\Ch^{1/2}:\Ch^{1/2}=\Ch$. The existence of such a tensor field follows from Lemma \ref{lem:Csqrt}. The numerical flux $\vu^*$  is defined in (\ref{eq:defFlux}), although the stability analysis in this paper holds for arbitrary linear flux operators. Note that (\ref{eq:coerCond2}) is satisfied when
\begin{align}
a_{e}(\vu,\vu) \geq c_{coer} \big| \vu \big|^2_{1,e}, &&e\in\mathcal{T}_h, \vu\in U_h,
\label{eq:coerCond1}
\end{align}
where $a_e(\vu,\vw):=a_e^{(C)}(\vu,\vw) + a_{\partial e}^{(DG)}(\vu,\vw) + a_{\partial e}^{(DG)}(\vw,\vu) + \eta_ea_{\partial e}^{(IP)}(\vu,\vw)$. Since we can write $|\vu|_{1,e}^2 = a_e^{(C)}(\vu,\vu) + \eta_ea_{\partial e}^{(IP)}(\vu,\vu)$, it remains to bound $a_{\partial e}^{(DG)}(\vu,\vu)$ in terms of $a_e^{(C)}(\vu,\vu)$ and $a_{\partial e}^{(IP)}(\vu,\vu)$. In order to do this, we first introduce the auxiliary bilinear form $a_{\partial e}^{(C*)}:U_h\times U_h\rightarrow\mathbb{R}$ defined by
\begin{align*}
a_{\partial e}^{(C*)}(\vu,\vw) &:= \int_{\partial e} (\nabla\vu^t):\nu_h^{-1}\Ch:\nabla\vw \;ds.
\end{align*}
Note that this operator is similar to $a_{e}^{(C)}$, but integrates over the element boundary instead of the interior. Next, we show that $a_{\partial e}^{(DG)}(\vu,\vu)$ can be bounded in terms of $a_{\partial e}^{(C*)}(\vu,\vu)$ and $a_{\partial e}^{(IP)}(\vu,\vu)$:

\begin{lem}
Consider an arbitrary element $e\in\mathcal{T}_{h}$, and let $c>0$ be an arbitrary positive constant. Then the following inequality holds:
\begin{align}
|2a^{(DG)}_{\partial e}(\vu,\vu)| &\leq c^{-1} a^{(C*)}_{\partial e}(\vu,\vu) + ca^{(IP)}_{\partial e}(\vu,\vu), &&\vu\in U_h.
\label{eq:ineqA2}
\end{align}
\label{lem:ineqA2}
\end{lem}

\begin{proof}
Take an arbitrary function $u\in U_h$. We can write
\begin{align*}
2a^{(DG)}_{\partial e}(\vu,\vu) &= \int_{\partial e} 2\Big( c^{1/2}\nu_h^{1/2}\Ch^{1/2}:\vn(\vu^*-\vu) , \;c^{-1/2}\nu_h^{-1/2}\Ch^{1/2}:\nabla\vu \Big) \;ds.
\end{align*}
Using the Cauchy--Schwarz and the Cauchy inequalities, we can then obtain
\begin{align*}
\big| 2a^{(DG)}_{\partial e}(\vu,\vu) \big| \leq&\;  c \int_{\partial e} \big\| \nu_h^{1/2}\Ch^{1/2}:\vn(\vu^*-\vu)  \big\|^2 \;ds  \\
&\; + c^{-1}  \int_{\partial e} \big\| \nu_h^{-1/2}\Ch^{1/2}:\nabla\vu  \big\|^2 \;ds \\
=&\; c^{-1} a^{(C*)}_{\partial e}(\vu,\vu) + ca^{(IP)}_{\partial e}(\vu,\vu).
\end{align*}
\end{proof}

We now construct the following constant:
\begin{align*}
\kappa_{e}^* &:= \sup_{\vu\in U_e, \;a^{(C)}_{e}(\vu,\vu)\neq 0 } \frac{ a_{{\partial e}}^{(C*)} (\vu,\vu) }{ a_{e}^{(C)}(\vu,\vu) },
\end{align*}
where $U_e$ is the discrete function space restricted to element $e$. From its definition, we can immediately obtain the inequality  $a_{\partial e}^{(C*)}(\vu,\vu)\leq \kappa_e^*a_e^{(C)}(\vu,\vu)$ for any $\vu\in U_h$. Using this property and Lemma \ref{lem:ineqA2} we can prove in Theorem \ref{thm:condPen1} that $\kappa_e^*$ is a sufficient lower bound for $\eta_e$ to ensure $a_e$ to be coercive. 

However, before we give this proof, we first show that $\kappa_{e}^*$ is well defined and show how it can be computed efficiently. To do this, consider an arbitrary $e\in\mathcal{T}_h$, and let $\{ \vw_i \}_{i=1}^n$ be a set of basis functions spanning $U_e$. Using these basis functions we can define positive semidefinite matrices $A_e,A^*_{\partial e}\in\mathbb{R}^{n\times n}_{sym}$ as follows:
\begin{subequations}
\begin{align}
[A_e]_{ij} &= a_e^{(C)}( \vw_i, \vw_j), &&i,j=1,\dots,n \\
[A_{\partial e}^*]_{ij} &= a_{\partial e}^{(C*)}( \vw_i, \vw_j), &&i,j=1,\dots,n.
\end{align}
\label{eq:defKappaAB}%
\end{subequations}
If matrix $A$ had been positive definite, then we could have obtained, using Lemma \ref{lem:largestEigValue}, the following relation:
\begin{align*}
\kappa_{e}^* &= \sup_{\vu\in U_e, \;a^{(C)}_{e}(\vu,\vu)\neq 0 } \frac{ a_{\partial e}^{(C*)} (\vu,\vu) }{ a_{e}^{(C)}(\vu,\vu) }
= \sup_{\vvu\in \mathbb{R}^n, \vvu\neq 0 } \frac{\vvu^tA^*_{\partial e}\vvu }{\vvu^tA_e\vvu }
= \lambda_{max} (A_e^{-1}A^*_{\partial e}),
\end{align*}
where $\lambda_{max} (A_e^{-1}A^*_{\partial e})$ denotes the largest eigenvalue of $A_e^{-1}A^*_{\partial e}$. However, the matrix $A_e$ is only positive semidefinite, so we need to obtain some intermediate results before we can show that a similar type of relation still holds. First, we show that the kernel of $A_e$ is a subset of the kernel of $A^*_{\partial e}$.

\begin{lem}
Let $e\in\mathcal{T}_h$ be an arbitrary element, and let $A_e,A^*_{\partial e}$ be matrices defined as in (\ref{eq:defKappaAB}). Then $\mathrm{Ker}(A_e)\subset\mathrm{Ker}(A^*_{\partial e})$.
\label{lem:kerAB}
\end{lem}

\begin{proof}
Let $\vvu\in\mathrm{Ker}(A_e)$, and define $\vu\in U_e$ as follows: $\vu := \sum_{i=1}^n \vvu_i\vw_i$. Then
\begin{align*}
0 &= \vvu^tA_e\vvu = a_e^{(C)}(\vu,\vu) = \int_e \| \Ch^{1/2}:\nabla\vu \|^2 \;dx.
\end{align*}
From this it follows that $\Ch:\nabla\vu = \tzero$ in $e$. Now let $\tilde\vu:=\vu\circ\phi_e$ be the reference function, with $\phi_e$ the reference-to-physical element mapping. Since  $\phi_e$ is assumed to be a polynomial function, such that $|\tJ_e| := |\mathrm{det}(\nabla\phi_e)| \geq |\tJ_e|_{min}> 0$ in $e$, for some constant $|\tJ_e|_{min}$, it follows that $\phi_e^{-1}\in C^{\infty}(e)^d$. Furthermore, since the reference function $\tilde\vu$ is also assumed to be polynomial, this implies that $\vu=\tilde\vu\circ\phi_e^{-1}\in C^{\infty}(e)^d$. Because $\Ch|_e\in W^{1,\infty}(e)^{d\times m\times m\times d}_{sym}$, it then follows from the trace theorem that $\Ch:\nabla\vu = \tzero$ is also satisfied on the boundary $\partial e$. This implies $a^{(C*)}_{\partial e}(\vu,\vw) = 0$ for any $\vw\in U_e$ and therefore $\vvu\in\mathrm{Ker}(A^*_{\partial e})$.
\end{proof}

Now let $k$ be the rank of $A_e$, and let $V_e\in\mathbb{R}^{n\times n}$ be a nonsingular matrix such that $V_e^tA_eV_e = D_e$, where $D_e\in\mathbb{R}^{n\times n}_{sym}$ is a diagonal matrix with the last $n-k$ entries being zero. Such a matrix decomposition can be obtained from, for example, a symmetric Gauss elimination procedure or a singular value decomposition. We then use these matrices to construct matrices $\tilde{D}_e, \tilde{A}^*_{\partial e}\in\mathbb{R}^{k\times k}_{sym}$ as follows:
\begin{subequations}
\begin{align}
[\tilde{D}_e]_{ij} &= [D_e]_{ij}, &&i,j=1,\dots,k, \\
[\tilde{A}^*_{\partial e}]_{ij} &= [V_e^tA^*_{\partial e}V_e]_{ij}, &&i,j=1,\dots,k.
\end{align}
\label{eq:defKappaDB}%
\end{subequations}
Using these matrices $\tilde{D}_e$ and $\tilde{A}^*_{\partial e}$ we can compute $\kappa_e^*$ and show that it is well defined.

\begin{lem}
Let $e\in\mathcal{T}_h$ be an arbitrary element, and let $\tilde{D}_e, \tilde{A}^*_{\partial e}$ be the matrices defined as in (\ref{eq:defKappaDB}). The constant $\kappa_e^*$ is well defined and satisfies
\begin{align}
\kappa_e^* = \lambda_{max}\big(\tilde{D}_e^{-1}\tilde{A}^*_{\partial e}\big),
\label{eq:propKappa}
\end{align}
where $\lambda_{max}\big(\tilde{D}_e^{-1}\tilde{A}^*_{\partial e}\big)$ denotes the largest eigenvalue of $\tilde{D}_e^{-1}\tilde{A}^*_{\partial e}$.
\label{lem:propKappa}
\end{lem}

\begin{proof}
First, consider the decomposition $V_e^tA_eV_e=D_e$ which was used to construct $\tilde{D}_e$ and $\tilde{A}^*_{\partial e}$. Since matrix $A_e$ has rank $k$ and the last $n-k$ entries of $D_e$ are zero, and since $V_e$ is nonsingular, this implies that the last $n-k$ columns of $V_e$ span the kernel of $A_e$. From Lemma \ref{lem:kerAB} it follows that these columns are also in the kernel of $A^*_{\partial e}$. Now let $\vvw\in\mathbb{R}^n$, and let $\tilde\vvw\in\mathbb{R}^k$ be the vector composed of the first $k$ entries of $\vvw$. We can then obtain the following relation:
\begin{align}
\vvw^tV_e^t A^*_{\partial e}V_e\vvw  &= \tilde{\vvw}^t\tilde{A}^*_{\partial e}\tilde{\vvw}.
\label{eq:propKappaB}
\end{align}
Since $A_e$ is positive semidefinite, it also follows that all entries of $\tilde{D}_e$ are strictly positive. Furthermore, since $A^*_{\partial e}$ is positive semidefinite, the matrix $\tilde{A}^*_{\partial e}$ will be positive semidefinite as well. Using these properties, we can prove (\ref{eq:propKappa}) as follows:
\begin{align*}
\kappa_e^* &:= \sup_{\vu\in U_e, \;a^{(C)}_{e}(\vu,\vu)\neq 0 } \frac{ a_{\partial e}^{(C*)} (\vu,\vu) }{ a_{e}^{(C)}(\vu,\vu) } 
= \sup_{\vvu\in\mathbb{R}^n, \vvu^tA_e\vvu\neq 0} \frac{\vvu^tA^*_{\partial e}\vvu}{\vvu^tA_e\vvu} \\
&= \sup_{\vvw\in\mathbb{R}^n, \vvw^tD_e\vvw\neq 0} \frac{\vvw^tV_e^tA^*_{\partial e}V_e\vvw}{\vvw^tD_e\vvw} 
= \sup_{\tilde{\vvw}\in\mathbb{R}^k, \tilde{\vvw}\neq \vvct{0}} \frac{\tilde{\vvw}^t\tilde{A}^*_{\partial e}\tilde{\vvw}}{\tilde{\vvw}^t\tilde{D}_e\tilde{\vvw}} 
= \lambda_{max}\big(\tilde{D}_e^{-1}\tilde{A}^*_{\partial e}\big).
\end{align*}
In the third step we substituted $\vvu$ by $V_e\vvw$, in the fourth step we used (\ref{eq:propKappaB}), and in the last step we used Lemma \ref{lem:largestEigValue} combined with the fact that $\tilde{D}_e$ is positive definite and $\tilde{A}^*_{\partial e}$ is positive semidefinite.
\end{proof}
\begin{rem} A symmetric Gauss elimination procedure or a singular value decomposition algorithm usually does not give the exact decomposition $V_e^tA_eV_e=D_e$, but only a numerical approximation. The diagonal entries of $D_e$ are then considered to be $0$ when they are smaller than a given tolerance.
\end{rem}
\begin{rem} The largest eigenvalue $\lambda_{max} \big(\tilde{D}_e^{-1}\tilde{A}^*_{\partial e}\big)$ can be efficiently obtained using a power iteration method.
\end{rem}

We can now derive the following sufficient estimate for the penalty term.

\begin{thm}
Let $e\in\mathcal{T}_{h}$ be an arbitrary element, and let $c_{\kappa}\geq1$ be an arbitrary constant. If $\eta_{e} \geq c_{\kappa}\kappa^*_{e}$, then $a_{e}(\vu,\vu) \geq 0$ for all $\vu\in U_h$. Moreover, if $c_{\kappa}>1$, then
\begin{align}
a_{e}(\vu,\vu) &\geq c_{coer} | \vu |^2_{1,e}, &&\vu\in U_h,
\label{eq:coercivity1}
\end{align}
where
\begin{align*}
c_{coer} &:= \sup_{x\in[1,c_{\kappa}]} \min\left\{ 1-x^{-1}, \frac{c_{\kappa} - x}{c_{\kappa}} \right\} > 0.
\end{align*}
\label{thm:condPen1}
\end{thm}

\begin{proof} Take an arbitrary function $\vu\in U_h$ and scalar $x\in[1,c_{\kappa}]$. We can then derive the following inequality:
\begin{align*}
a_{e}(\vu,\vu) =&\; a^{(C)}_{e}(\vu,\vu) + 2a^{(DG)}_{\partial e}(\vu,\vu) + \eta_{e}a^{(IP)}_{\partial e}(\vu, \vu) \\
\geq&\; a^{(C)}_{e}(\vu,\vu) - x^{-1}(\kappa^*_{e})^{-1}a^{(C*)}_{\partial e}(\vu,\vu) + (\eta_{e}-x\kappa^*_{e})a^{(IP)}_{\partial e}(\vu, \vu) \\
\geq &\;  (1-x^{-1})a^{(C)}_{e}(\vu,\vu) +  (c_{\kappa}-x)\kappa^*_{e}a^{(IP)}_{\partial e}(\vu,\vu).
\end{align*}
In the second line we used Lemma \ref{lem:ineqA2} with $c=x\kappa^*_{e}$, and in the last line we used the definition of $\kappa^*_{e}$. Now note that we can write $| \vu |^2_{1,e} = a^{(C)}_{e}(\vu,\vu) + c_{\kappa}\kappa^*_{e}a^{(DG)}_{\partial e}(\vu,\vu)$. Combining this with the inequality above gives
\begin{align*}
a_{e}(\vu,\vu) \geq \min\left\{ 1-x^{-1}, \frac{c_{\kappa} - x}{c_{\kappa}} \right\} | \vu |^2_{1,e} \geq 0.
\end{align*}
Taking the supremum over all $x\in[1,c_{\kappa}]$ results in (\ref{eq:coercivity1}).
\end{proof}

The penalty term estimate depends on the constant $\kappa^*_{e}$. However, this constant does not include any effects of the normal vector on the positivity of the bilinear operator, which may cause the penalty term estimate to be less sharp. Therefore, we consider an additional penalty term estimate which does include the effect of the normal vector, and is shown to be considerably sharper in Section \ref{sec:numResults}. To do this, we first define the tensor field $\ten{c}_{\vn}\in \bigotimes_{e\in\mathcal{T}_h}L^{\infty}(\partial e)^{m\times m}_{sym}$ as follows:
\begin{align*}
\ten{c}_{\vn} |_{\partial e} &:= (\vn\cdot \Ch \cdot\vn)|_{\partial e}, && e\in\mathcal{T}_h.
\end{align*}
where $\vn|_{\partial e}$ is the outward pointing normal vector of element $e$. We also define the following function space:
\begin{align*}
\hat U_h := \Big\{ \hat\vu\in\bigotimes_{e\in\mathcal{T}_h}L^2(\partial e)^m \;\Big|\; &\hat\vu|_{\partial e} = (\vn\cdot \Ch:\nabla\vu)|_{\partial e}, \\
&\text{ for some }u\in U_h, \text{ for all }e\in\mathcal{T}_h \Big\}.
\end{align*}
Lemma \ref{lem:CnInv} shows that there exists a pseudoinverse $\ten{c}_{\vn}^{-1}\in\bigotimes_{e\in\mathcal{T}_h}L^{\infty}(\partial e)^{m\times m}_{sym}$ such that $\ten{c}_{\vn}^{-1} \cdot \ten{c}_{\vn}\cdot\hat\vu = \ten{c}_{\vn} \cdot \ten{c}_{\vn}^{-1}\cdot\hat\vu = \hat\vu$ for all $\hat\vu \in \hat U_h$. We use this tensor field to define an alternative auxiliary bilinear operator $a^{(C**)}_{\partial e}:U_h\times U_h\rightarrow\mathbb{R}$ as follows:
\begin{align*}
a^{(C**)}_{\partial e}(\vu,\vw) &:=  \int_{\partial e} (\vn\cdot \Ch:\nabla\vu)^t \cdot \nu_h^{-1}\ten{c}_{\vn}^{-1} \cdot (\vn\cdot \Ch:\nabla\vw\big) \;ds.
\end{align*} 
The penalty term estimate and the coercivity result are obtained in the same way as before, except that we now use $a^{(C**)}_{\partial e}$ instead of $a^{(C*)}_{\partial e}$. We start again by deriving a bound on $a^{(DG)}_{\partial e}$:

\begin{lem}
Consider an arbitrary element $e\in\mathcal{T}_{h}$, and let $c>0$ be an arbitrary positive constant. Then the following inequality holds:
\begin{align}
|2a^{(DG)}_{\partial e}(\vu,\vu)| &\leq c^{-1} a^{(C**)}_{\partial e}(\vu,\vu) + ca^{(IP)}_{\partial e}(\vu,\vu), &&\vu\in U.
\label{eq:ineqA2_2}
\end{align}
\label{lem:ineqA2_2}
\end{lem}

\begin{proof}
From Lemma \ref{lem:CnInv} it follows that $\ten{c}_{\vn}$ and $\ten{c}_{\vn}^{-1}$ are positive semidefinite tensor fields, and therefore there exist symmetric positive semidefinite tensor fields $\ten{c}_{\vn}^{1/2}, \ten{c}_{\vn}^{-1/2}\in\bigotimes_{e\in\mathcal{T}_h}L^{\infty}(\partial e)^{m\times m}_{sym}$ such that $\ten{c}_{\vn}^{1/2} \cdot \ten{c}_{\vn}^{1/2} = \ten{c}_{\vn}$ and $\ten{c}_{\vn}^{-1/2} \cdot \ten{c}_{\vn}^{-1/2} = \ten{c}_{\vn}^{-1}$, and such that $\ten{c}_{\vn}^{-1/2} \cdot \ten{c}_{\vn}^{1/2} \cdot \hat\vu = \ten{c}_{\vn}^{1/2} \cdot \ten{c}_{\vn}^{-1/2} \cdot \hat\vu = \hat\vu$ for all $\hat\vu\in \hat U_h$.

Now take an arbitrary function $u\in U_h$, and define the function $\hat\vu\in \hat U_h$ as follows:
\begin{align*}
\hat\vu|_{\partial e} &:= (\vn\cdot \Ch:\nabla\vu)|_{\partial e}, && e\in\mathcal{T}_h.
\end{align*}
We can then write
\begin{align*}
2a^{(DG)}_{\partial e}(\vu,\vu) &= \int_{\partial e} 2\Big( c^{1/2}\nu_h^{1/2}(\vu^*-\vu) , \;c^{-1/2}\nu_h^{-1/2}\hat\vu \Big) \;ds \\
  &=\int_{\partial e} 2\Big( c^{1/2}\nu_h^{1/2}(\vu^*-\vu) , \;c^{-1/2}\nu_h^{-1/2}\ten{c}_{\vn}^{1/2}\cdot\ten{c}_{\vn}^{-1/2}\cdot\hat\vu \Big) \;ds \\
  &= \int_{\partial e} 2\Big( c^{1/2}\nu_h^{1/2}\ten{c}_{\vn}^{1/2}\cdot(\vu^*-\vu) , \;c^{-1/2}\nu_h^{-1/2}\ten{c}_{\vn}^{-1/2}\cdot\hat\vu \Big) \;ds.
\end{align*}
Using the Cauchy--Schwarz and the Cauchy inequalities, we can then obtain
\begin{align*}
\big| 2a^{(DG)}_{\partial e}(\vu,\vu) \big| \leq&\;  c\int_{\partial e} \big\| \nu_h^{1/2}\ten{c}_{\vn}^{1/2} \cdot (\vu^*-\vu)  \big\|^2 \;ds  \\
&\; + c^{-1}  \int_{\partial e} \big\| \nu_h^{-1/2}\ten{c}_{\vn}^{-1/2}\cdot\hat\vu \big\|^2 \;ds \\
=&\; c^{-1} a^{(C**)}_{\partial e}(\vu,\vu) + ca^{(IP)}_{\partial e}(\vu,\vu).
\end{align*}
\end{proof}

We now use the bilinear operator $a_{\partial e}^{(C**)}$ to construct the following constant:
\begin{align*}
\kappa_{e}^{**} &:= \sup_{\vu\in U_e, \;a^{(C)}_{e}(\vu,\vu)\neq 0 } \frac{ a_{\partial e}^{(C**)} (\vu,\vu) }{ a_{e}^{(C)}(\vu,\vu) }.
\end{align*}
The proof of the existence of this constant and the way to compute it is analogous to $\kappa_e^*$. In a similar way as before we can use this constant to obtain the following sufficient penalty term estimate.

\begin{thm}
Let $e\in\mathcal{T}_{h}$ be an arbitrary element, and let $c_{\kappa}\geq1$ be an arbitrary constant. If $\eta_{e} \geq c_{\kappa}\kappa^{**}_{e}$, then $a_{e}(\vu,\vu) \geq 0$ for all $\vu\in U_h$. Moreover, if $c_{\kappa}>1$, then 
\begin{align}
a_{e}(\vu,\vu) &\geq c_{coer} | \vu |^2_{1,e}, &&\vu\in U_h,
\label{eq:coercivity2}
\end{align}
where
\begin{align*}
c_{coer} &:= \sup_{x\in[1,c_{\kappa}]} \min\left\{ 1-x^{-1}, \frac{c_{\kappa} - x}{c_{\kappa}} \right\}.
\end{align*}
\label{thm:condPen2}
\end{thm}

\begin{proof} The proof is analogous to that of Theorem \ref{thm:condPen1}.
\end{proof}

We have now derived conditions for the penalty term to ensure that $a_{h}$ is positive semidefinite and showed how the penalty term can be computed. In the next section we will derive time step estimates to ensure that the local time-stepping scheme is stable.

\section{Sufficient time step estimates}
\label{sec:timeStepEst}
We start by rewriting the DG method as a linear system of ordinary differential equations. We then show how we can obtain sufficient upper bounds for the spectral radius of $M^{-1}A$, and therefore sufficient lower bounds for the time step size for a large class of explicit time integration schemes, by splitting the mass mass matrix $M$ and stiffness matrix $A$ into multiple parts. Finally, we introduce weighted mesh decompositions to explain how this splitting of matrices can be done efficiently. 

\subsection{A system of ordinary differential equations}
Let $\{\vw_i\}_{i=1}^{N}$ be a linear basis of $U_h$ and define, for $\vu\in U_h$, the vector $\vvct u\in\mathbb{R}^{N}$ such that $\vu = \sum_{i=1}^{N} \vvct u_i\vw_i$. We can rewrite the DG method, given in (\ref{eq:DGmethod}), as the following system of ordinary differential equations: we solve $\vvu:[0,T]\rightarrow\mathbb{R}^N$, such that
\begin{subequations}
\begin{align}
M_h\partial_t^2\vvct u + A_h\vvct u &= \vvct f^*_h, &&t\in[0,T], \\
\vvu|_{t=0} &= \vvu_{0,h} := M_h^{-1}\vvct{u}_{0,h}^*, && \\
\partial_t\vvu|_{t=0} &= \vvct v_{0,h} := M_h^{-1}\vvct{v}_{0,h}^*, &&
\end{align}
\label{eq:ODEsystem}%
\end{subequations}
where $M_h, A_h\in\mathbb{R}^{N\times N}$ are matrices,  $\vvct{u}_{0,h}^*, \vvct{v}_{0,h}^*\in\mathbb{R}^N$ are vectors, and $ \vvct f_h^*:[0,T]\rightarrow\mathbb{R}^N$ is a vector function, defined as follows:
\begin{align*}
[M_h]_{ij} &:= (\rho\vw_i,\vw_j), &&i,j=1,\dots,N,\\
[A_h]_{ij} &:= a_h(\vw_i,\vw_j), &&i,j=1,\dots,N, \\
[\vvct{u}_{0,h}^*]_i &:= (\rho\vu_{0}, \vw_i), &&i=1,\dots,N, \\
[\vvct{v}_{0,h}^*]_i &:= (\rho \vv_{0}, \vw_i), &&i=1,\dots,N, \\
[\vvct f_h^*(t)]_i &:= (\vf(t), \vw_i), && i=1,\dots,N, \;t\in[0,T].
\end{align*}
For a large class of explicit time integrators, including Lax--Wendroff schemes and explicit Runge--Kutta schemes, the time step size condition is of the form
\begin{align}
\Delta t \leq \frac{c_{method}}{\sqrt{\lambda_{max}(M^{-1}A)}},
\label{eq:timeStepCond1}
\end{align} 
where $c_{method} > 0$ is a constant, depending only on the type of time integration method, and $\lambda_{max}(M^{-1}A)$ is the largest eigenvalue of $M^{-1}A$, which is also known as the spectral radius of $M^{-1}A$. For example, the stability condition for the leap--frog scheme is well known to be (\ref{eq:timeStepCond1}) with $c_{method}=2$. Because of the form of (\ref{eq:timeStepCond1}), it remains to find an upper estimate for the spectral radius. In the next section we show how this can be done by splitting the matrices $M$ and $A$ into multiple parts.

\subsection{Spectral radius estimates by splitting matrices}
In order to obtain a bound for the spectral radius we first introduce the mapping $\mathcal{I}:\mathbb{R}^{N\times N}_{sym}\rightarrow\mathbb{R}^{N\times N}_{sym}$, which maps a symmetric matrix to a diagonal matrix with entries $0$ and $1$ to indicate the nonzero rows or columns of the input matrix:
\begin{align*}
[\mathcal{I}(S)]_{ij} &= \begin{cases}
1 & i=j, \text{ and } S_{ik}\neq 0 \text{ for somes }k\in\{1,..,N\}, \\
0 & \text{otherwise}.
\end{cases} 
\end{align*}
We also define $\mathcal{I}^*(S)$ as the matrix $\mathcal{I}(S)$ with all zero-columns removed. Using these definitions we can formulate the following theorem.

\begin{thm}
Let $M\in\mathbb{R}^{N\times N}_{sym}$ be a symmetric positive definite matrix and $A\in\mathbb{R}^{N\times N}_{sym}$ a symmetric positive semidefinite matrix. Also let $M_{(i)}, A_{(i)} \in\mathbb{R}^{N\times N}_{sym}$ for $i=1,\dots,n$, be symmetric matrices such that
\begin{subequations}
\begin{align}
\sum_{i=1}^n M_{(i)} &= M, \qquad \sum_{i=1}^n A_{(i)} = A, &&\\
M_{(i)}' &\succ 0, \qquad I_{(i)}A_{(i)}I_{(i)} = A_{(i)}, && i=1,\dots,n, 
\end{align}
\label{eq:matrixDec}%
\end{subequations}
where $M_{(i)}':=(I_{(i)}^*)^tM_{(i)}I^*_{(i)}$, $I_{(i)} :=\mathcal{I}(M_{(i)})$ and $I_{(i)}^* :=\mathcal{I}^*(M_{(i)})$. Then
\begin{align}
\lambda_{max}(M^{-1}A) &\leq \max_{i=1,\dots,n} \lambda_{max}\big((M_{(i)}')^{-1}A_{(i)}'\big),
\label{eq:spectralRadius}
\end{align}
where $A_{(i)}':=(I^*_{(i)})^tA_{(i)}I^*_{(i)}$, and $\lambda_{max}(\cdot)$ denotes the largest eigenvalue in magnitude.
\label{thm:spectralRadius}
\end{thm}

\begin{rem} The matrices $M'_{(i)}$ and $A'_{(i)}$ are the submatrices of $M_{(i)}$ and $A_{(i)}$, respectively, obtained by removing all rows and columns corresponding to the zero rows and columns of $M_{(i)}$. The condition $I_{(i)}A_{(i)}I_{(i)} = A_{(i)}$ means that any zero column or row of $M_{(i)}$ is also a zero column or row of $A_{(i)}$, and the condition $M_{(i)}' \succ 0$ means that the submatrices of $M$ are positive definite.
\end{rem}

\begin{proof}
For any $\vvu\in\mathbb{R}^n$, define the following set of indices:
\begin{align*}
\mathbb{I}(\vvu) &:= \big\{i\in\{1,\dots,n\}\;\big|\; I_{(i)}\vvu \neq \vvct{0} \big\}.
\end{align*}
Using Lemma \ref{lem:largestEigValue} and Lemma \ref{lem:ineqFrac} we can then bound the largest eigenvalue as follows:
\begin{align*}
\lambda_{max}(M^{-1}A) &= \sup_{\vvu\in\mathbb{R}^{N}, \;\vvu\neq \vvct{0} } \frac{\vvu^tA\vvu}{ \vvu^tM\vvu} \\
&= \sup_{\vvu\in\mathbb{R}^{N}, \;\vvu\neq \vvct{0} } \frac{\sum_{i=1}^n  \vvu^tA_{(i)}\vvu}{\sum_{i=1}^n  \vvu^tM_{(i)}\vvu} 
= \sup_{\vvu\in\mathbb{R}^{N}, \;\vvu\neq \vvct{0} } \frac{\sum_{i\in\mathbb{I}(\vvu)} \vvu^tA_{(i)}\vvu}{\sum_{i\in\mathbb{I}(\vvu)}  \vvu^tM_{(i)}\vvu} \\
&\leq \sup_{\vvu\in\mathbb{R}^{N}, \;\vvu\neq \vvct{0} } \;\max_{ i\in\mathbb{I}(\vvu) } \frac{ |\vvu^tA_{(i)}\vvu|}{ \vvu^tM_{(i)}\vvu} 
=  \max_{i=1,\dots,n} \;\sup_{\vvu\in\mathbb{R}^{N}, \;I_{(i)}\vvu\neq \vvct{0} } \frac{ |\vvu^tA_{(i)}\vvu|}{ \vvu^tM_{(i)}\vvu}.
\end{align*}
Using Lemma \ref{lem:largestEigValue} again, we can obtain, for any $i=1,..,n$, the following:
\begin{align*}
\sup_{\vvu\in\mathbb{R}^{N}, \;I_{(i)}\vvu\neq \vvct{0} } \frac{| \vvu^tA_{(i)}\vvu|}{ \vvu^tM_{(i)}\vvu} &= \sup_{\vvu\in\mathbb{R}^{N_i}, \;\vvu\neq \vvct{0} } \frac{ |\vvu^tA_{(i)}'\vvu|}{ \vvu^tM_{(i)}'\vvu} 
= \lambda_{max}\big((M_{(i)}')^{-1}A_{(i)}'\big),
\end{align*}
where $N_i$ is the number of nonzero columns of $M_{(i)}$. Combining these results gives (\ref{eq:spectralRadius}).
\end{proof}

To apply the above theorem it remains to find a decomposition of the matrices $M$ and $A$ such that (\ref{eq:matrixDec}) is satisfied. For continuous finite elements such a decomposition can be easily obtained from the element matrices,
\begin{align*}
M=\sum_{e\in\mathcal{T}_h} M_e, \qquad A=\sum_{e\in\mathcal{T}_h} A_e,
\end{align*}
where $M_e$ and $A_e$ are the element matrices corresponding to the mass matrix $M$ and stiffness matrix $A$, respectively. Using Theorem \ref{thm:spectralRadius} we then obtain the following estimate for the spectral radius:
\begin{align*}
\lambda_{max}(M^{-1}A) &\leq \max_{e\in\mathcal{T}_h} \lambda_{max}\big((M_{e}')^{-1}A_{e}'\big),
\end{align*}
where $M_{e}' := (I_{e}^{*})^tM_{e}I_{e}^*$, $A_{e}' := (I_{e}^{*})^tA_{e}I_{e}^*$ and $I_{e}^* :=\mathcal{I}^*(M_{e})$. In other words, the largest eigenvalue of the global matrix is bounded by the supremum over all elements of the largest eigenvalue of the element matrix. This result was already mentioned by \cite{irons71}. For discontinuous elements, however, a suitable decomposition of the matrices is less straightforward due to the face integral terms. In the next subsection we show how we can decompose the matrices for discontinuous elements, using a weighted mesh decomposition.

\subsection{A weighted mesh decomposition}
\label{sec:weightedDomDec}
We define a weighted submesh $\omega:\mathcal{T}_h\cup\mathcal{F}_h\rightarrow[0,1]$ to be a function that assigns to every element and face a weight value $\omega_e$ and $\omega_f$ between $0$ and $1$, such that if $\omega_f>0$ for a certain face $f$, then $\omega_e>0$ for the adjacent elements $e\in\mathcal{T}_f$. We call a set of weighted submeshes $\mathcal{W}_h$ a weighted mesh decomposition of $\mathcal{T}_h$ if the sum of all weighted submeshes adds up to one for every face and element: $\sum_{\omega\in\mathcal{W}_h} \omega_e = 1$ for all $e\in\mathcal{T}_h$ and $\sum_{\omega\in\mathcal{W}_h} \omega_f = 1$ for all $f\in\mathcal{F}_h$. An illustration of a weighted mesh decomposition is given in Figure \ref{fig:weightedMeshDec}.

\begin{figure}[h]
\centering
\includegraphics[width=0.8\textwidth]{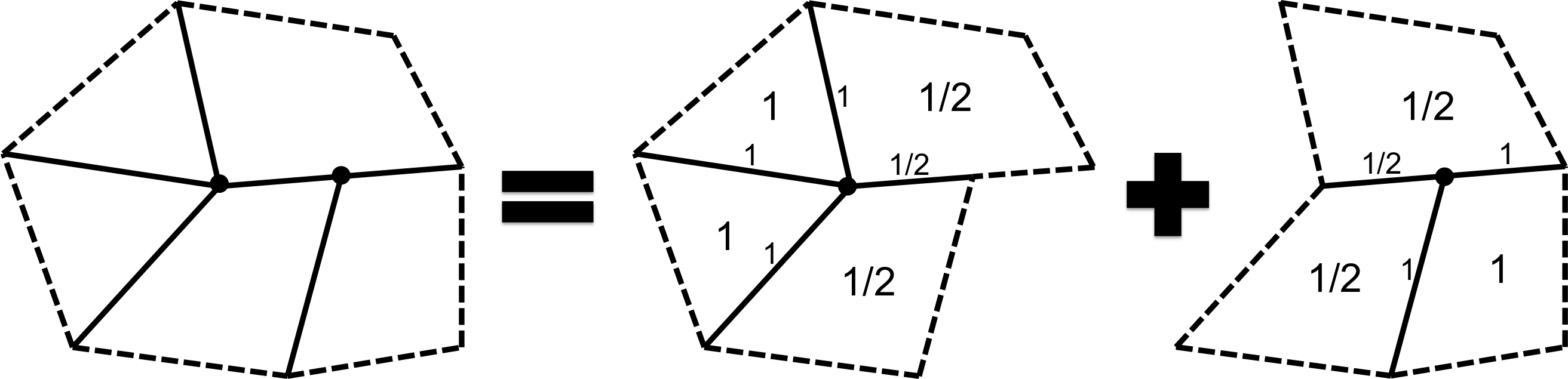}
\caption{A weighted mesh decomposition. The larger numbers denote the element weights, while the smaller numbers denote the face weights. Weight values of elements and faces outside the illustrated subdomains are zero. }
\label{fig:weightedMeshDec}
\end{figure}

We can use a weighted submesh to construct bilinear forms $(\cdot,\cdot)_{\omega}, a_{\omega}:U_h\times U_h\rightarrow\mathbb{R}$ as follows:
\begin{align*}
(\vu,\vw)_{\omega} &:= \sum_{e\in\mathcal{T}_h} \omega_e \int_e \rho\vu\cdot\vw \;dx, \\
a_{\omega}(\vu,\vw) &:= a_{\omega}^{(C)}(\vu,\vw) + a_{\omega}^{(DG)}(\vu,\vw) + a_{\omega}^{(DG)}(\vw,\vu) + a_{\omega}^{(IP)}(\vu,\vw) 
\end{align*}
with
\begin{align*}
a_{\omega}^{(C)}(\vu,\vw) &:=  \sum_{e\in\mathcal{T}_{h}} \omega_e \int_{e} (\nabla\vu)^t:\Ch:\nabla\vw \;dx, \\
a_{\omega}^{(DG)}(\vu,\vw) &:=  \sum_{e\in\mathcal{T}_{h}}\sum_{f\in\mathcal{F}_{e}} \omega_f \int_{\partial e\cap f} (\vu^*-\vu)\vn:\Ch:\nabla\vw \;ds, \\
a_{\omega}^{(IP)}(\vu,\vw) &:=  \sum_{e\in\mathcal{T}_{h}}\sum_{f\in\mathcal{F}_{e}} \omega_f\eta_e \int_{\partial e\cap f}(\vu^*-\vu)\vn:\nu_h \Ch:\vn(\vw^*-\vw) \; ds. 
\end{align*}
Note that $(\vu,\vw) = \sum_{\omega\in\mathcal{W}_h} (\vu,\vw)_{\omega}$ and $a_h(\vu,\vw) = \sum_{\omega\in\mathcal{W}_h} a_{\omega}(\vu,\vw)$, for all $\vu,\vw\in U_h$.

For the numerical tests, we will in particular consider a weighted mesh decomposition based on the vertices, as illustrated in Figure \ref{fig:nodalDomDec}. The vertex-based mesh decomposition is given by $\mathcal{W}_h:= \{ \omega^{(q)} \}_{q\in\mathcal{Q}}$, with
\begin{align}
\omega^{(q)}_{e} := \begin{cases}
\frac{1}{|\mathcal{Q}_e|} & e\in\mathcal{T}_q, \\
0 & \text{otherwise},
\end{cases} \qquad
\omega^{(q)}_{f} := \begin{cases}
\frac{1}{|\mathcal{Q}_f|} & e\in\mathcal{F}_q, \\
0 & \text{otherwise},
\end{cases}
\label{eq:nodalDomDec}%
\end{align}
where $|\mathcal{Q}_e|$, $|\mathcal{Q}_f|$ are the number of vertices adjacent to element $e$ and face $f$, respectively, and $\mathcal{T}_q$, $\mathcal{F}_q$ are the set of elements and faces adjacent to $q$, respectively.

\begin{figure}[h]
\centering
\includegraphics[width=0.6\textwidth]{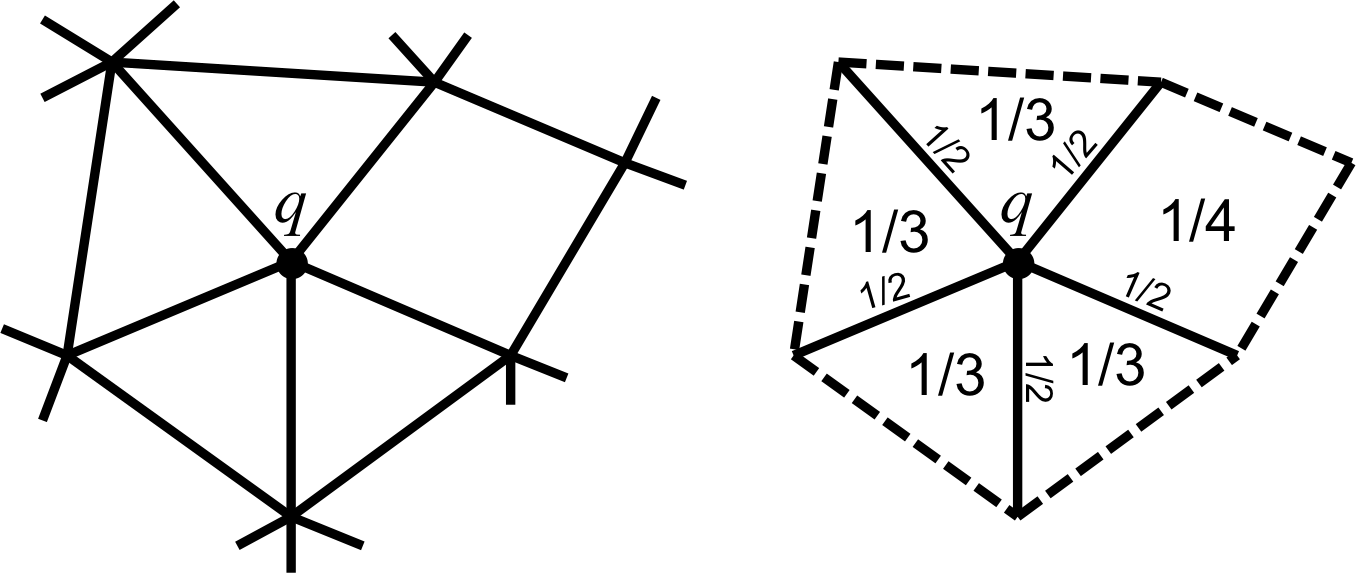}
\caption{Illustration of a vertex-based mesh decomposition. For every vertex $q$, a weighted submesh $\omega^{(q)}$ is created assigning nonzero values only for elements and faces directly adjacent to the vertex. }
\label{fig:nodalDomDec}
\end{figure}

Now let $\{\vw_i\}_{i=1}^N$ be a linear basis of $U_h$, such that every basis function is nonzero on only a single element $e_i$. We can use a weighted mesh decomposition $\mathcal{W}_h$ to decompose the mass matrix and stiffness matrix as follows:
\begin{align*}
M = \sum_{\omega\in\mathcal{W}_h} M_{\omega}, \qquad A=\sum_{\omega\in\mathcal{W}_h} A_{\omega},
\end{align*}
where $[M_{\omega}]_{ij} := (\vw_i,\vw_j)_{\omega}$ and $[A_{\omega}]_{ij} := a_{\omega}(\vw_i,\vw_j)$, for $i,j=1,\dots,N$. Using Theorem \ref{thm:spectralRadius} we can immediately obtain the following estimate for the spectral radius and therefore the time step size.
\begin{thm}
Let $\mathcal{W}_h$ be a weighted mesh decomposition. Then
\begin{align}
\lambda_{max}(M^{-1}A) &\leq \max_{\omega\in\mathcal{W}_h} \lambda_{max}\big((M_{\omega}')^{-1}A_{\omega}'\big),
\label{eq:spectralRadiusDG}
\end{align}
where $M_{\omega}' := (I_{\omega}^{*})^tM_{\omega}I_{\omega}^*$,  $A_{\omega}' := (I_{\omega}^{*})^tA_{\omega}I_{\omega}^*$, and $I_{\omega}^* :=\mathcal{I}^*(M_{\omega})$, and where $\lambda_{max}(\cdot)$ denotes the largest eigenvalue in magnitude.
\label{thm:spectralRadiusDG}
\end{thm}

\begin{rem} When the weighted submeshes $\omega$ are nonzero for only a few elements and faces, then $M'_{\omega}$ and $A'_{\omega}$ are relatively small matrices. The largest eigenvalue $ \lambda_{max}\big((M_{\omega}')^{-1}A_{\omega}'\big)$ can then be efficiently computed in parallel for each submatrix, using a power iteration method requiring only a relatively small number of iterations.
\end{rem}

In the next section we show several numerical results illustrating the sharpness of the penalty term and time step estimates.

\section{Numerical results}
\label{sec:numResults}

\subsection{Computing the spectral radius for periodic meshes}
\label{sec:periodicMesh}
To test the sharpness of the penalty term estimates and time step estimates we consider a $d$-dimensional cubic domain of the form $(0,N)^d$ with periodic boundary conditions. We then create a uniform mesh of $N^d$ unit cubes, after which we subdivide every cube into smaller elements and choose basis function sets and material parameters for every subelement. These subelements, basis functions and material parameters are chosen identically for every cube. An illustration of such a mesh is given in Figure \ref{fig:semiUniformMesh2D}. The advantage of such a uniform periodic mesh is that we can rather easily obtain the exact spectral radius by using a Fourier analysis, in a way similar to the von Neumann method for finite difference schemes.

\begin{figure}[h]
\centering
\includegraphics[width=0.6\textwidth]{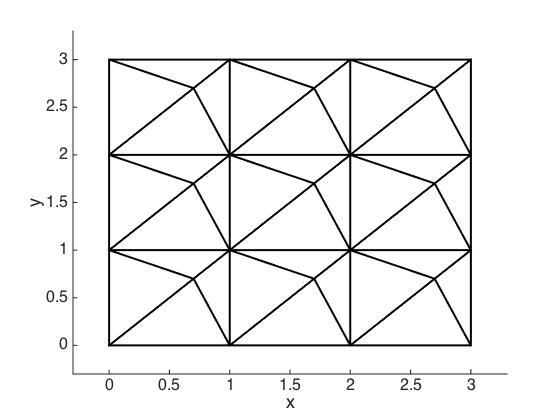}
\caption{Square 2D mesh consisting of $3^2$ unit cubes, where each square is identically subdivided into four distorted triangles.}
\label{fig:semiUniformMesh2D}
\end{figure}

To apply a Fourier analysis we first choose a linear basis $\{\vw_i\}_{i=1}^{N^d\times M}$ of the discrete function space $U$, such that the linear basis is of the form 
\begin{align*}
\{\vw_i\}_{i=1}^{N^d\times M} = \bigcup_{\mathbf{k}\in\mathbb{Z}_N^d} \{\vw_{\mathbf{k}, i}\}_{i=1}^M,
\end{align*}
where $\mathbf{k}\in\mathbb{Z}_N^d$ is the identifier of unit cube $(k_1,k_1+1)\times \cdots\times (k_d,k_d+1)$ and $ \{\vw_{\mathbf{k}, i}\}_{i=1}^M$ is a linear basis of the discrete space $U$ restricted to this cube. We can then define the submatrices $M_{\mathbf{k},\mathbf{l}}, A_{\mathbf{k},\mathbf{l}} \in\mathbb{R}^{M\times M}$ as follows:
\begin{align*}
[M_{\mathbf{k},\mathbf{l}}]_{ij} &:= (\rho \vw_{\mathbf{k},i}, \vw_{\mathbf{l},j}), &&i,j=1,\dots,M, \;\mathbf{k}, \mathbf{l}\in\mathbb{Z}_N^d,  \\
[A_{\mathbf{k},\mathbf{l}}]_{ij} &:= a_h(\rho \vw_{\mathbf{k},i}, \vw_{\mathbf{l},j}), &&i,j=1,\dots,M, \;\mathbf{k}, \mathbf{l}\in\mathbb{Z}_N^d,
\end{align*}
By construction of the mesh, most of these submatrices are identical. Fix any $\mathbf{l}\in\mathbb{Z}_N^d$. Then the submatrices $M_{\mathbf{k},\mathbf{k}+\mathbf{l}}$ are identical for any $\mathbf{k}\in\mathbb{Z}_N^d$. The same holds for $A_{\mathbf{k},\mathbf{k}+\mathbf{l}}$. Moreover, by definition of the mass matrix, $M_{\mathbf{k},\mathbf{k}+\mathbf{l}}$ is only nonzero when $\mathbf{l}=\mathbf{0}$, and by construction of the stiffness matrix, $A_{\mathbf{k},\mathbf{k}+\mathbf{l}}$ is only nonzero when $|\mathbf{l}|\leq 1$. Therefore, we only have to consider the submatrices $M_{0} := M_{\mathbf{k}, \mathbf{k}}$, $A_{0} := A_{\mathbf{k},\mathbf{k}}$, and $A_i^\pm := A_{\mathbf{k},\mathbf{k}\pm\mathbf{e}_i}$ for $i=1,\dots,d$, where $\mathbf{k}$ is an arbitrary vector in $\mathbb{Z}_{N}^d$ and $\mathbf{e}_i$ is the unit vector in direction $i$.

Now let $\vvu\in\mathbb{R}^{N^d\times M}$ be a vector of coefficients, and let $\vvu_{\mathbf{k}}\in\mathbb{R}^M$ be the vector of coefficients corresponding to cube $\mathbf{k}$. Suppose that $\vvw=M^{-1}A\vvu$. We can then write
\begin{align*}
\vvw_{\mathbf{k}} &= M^{-1}_0\left( A_0\vvu_{\mathbf{k}} + \sum_{i=1}^d (A_i^+\vvu_{\mathbf{k}+\mathbf{e}_i} + A_i^-\vvu_{\mathbf{k}-\mathbf{e}_i}) \right), &&\mathbf{k}\in\mathbb{Z}_N^d.
\end{align*}
Define $\vvu^{(\mathbf{z})}\in\mathbb{R}^{N^d\times M}$ as follows:
\begin{align}
\vvu_{\mathbf{k}}^{(\mathbf{z})} =  e^{\im(\mathbf{z}\cdot\mathbf{k}/N)2\pi}\vvu_{0}, &&\mathbf{k}\in\mathbb{Z}_N^d,
\label{eq:defu_z}
\end{align}
where $\vvu_0\in\mathbb{R}^M$ is an arbitrary vector of coefficients corresponding to a single cube, $\im:=\sqrt{-1}$ is the imaginary number, and $\mathbf{z}\in\mathbb{Z}_N^d$ is a vector of integers. Then $\vvw^{(\mathbf{z})}:=M^{-1}A\vvu^{(\mathbf{z})}$ satisfies
\begin{align}
\vvw_{\mathbf{k}}^{(\mathbf{z})} &= e^{\im(\mathbf{z}\cdot\mathbf{k}/N)2\pi} Z^{(\mathbf{z})} \vvu_0, && \mathbf{k}\in\mathbb{Z}_N^d,
\label{eq:propw_z}
\end{align} 
where
\begin{align*}
Z^{(\mathbf{z})} &:=  M^{-1}_0\left( A_0 + \sum_{i=1}^d (e^{\im(z_i/N)2\pi}A_i^+ + e^{-\im(z_i/N)2\pi}A_i^-)\right), && \mathbf{z}\in\mathbb{Z}_N^d.
\end{align*}
From (\ref{eq:defu_z}) and (\ref{eq:propw_z}) it follows that if $(\lambda, \vvu_0)$ is an eigenpair of $Z^{(\mathbf{z})}$, then $(\lambda, \vvu^{(\mathbf{z})})$ is an eigenpair of $M^{-1}A$. Since $Z^{(\mathbf{z})}$ has $M$ eigenpairs and since there are $N^d$ possible choices for $\mathbf{z}$, every eigenvalue of $M^{-1}A$ is an eigenvalue of $Z^{(\mathbf{z})}$ for some $\mathbf{z}\in\mathbb{Z}_N^d$. For the time step estimates we are only interested in the largest eigenvalue $\lambda_{max}(M^{-1}A)$, which we can then compute by
\begin{align*}
\lambda_{max}(M^{-1}A) &= \sup_{\mathbf{z}\in\mathbb{Z}_N^d}\lambda_{max}(Z^{(\mathbf{z})} ).
\end{align*}
For the numerical tests that we will present here, we have taken $N=2$, since in most cases the largest eigenvalue $\lambda_{max}(M^{-1}A)$ no longer increases significantly for $N>2$.

\subsection{Sharpness of the penalty term and time step estimates}
For testing the sharpness of our parameter estimates we use polynomial basis functions up to degree $p$ for simplicial elements, and polynomials up to degree $p$ in the direction of each reference coordinate for quadrilateral and hexahedral elements. First, we consider several regular homogeneous meshes for the acoustic wave equation in 1D, 2D and 3D. After that we test on meshes with deformed elements and meshes with piecewise linear parameter fields. We also test on meshes for electromagnetic and elastic wave problems, including heterogeneous meshes with sharp material contrasts and meshes with sharp contrasts in primary and secondary wave velocities.

To test the sharpness of the parameters we first compute the penalty terms as in Theorem \ref{thm:condPen1} or Theorem \ref{thm:condPen2} with $c_{\kappa}=1$. We will refer to the first penalty terms as $\eta_e^{*}$ and to the second as $\eta_e^{**}$. We then find the smallest scale $c_{min}\in[0,1]$ such that the stiffness matrix $A$ is still positive semidefinite when using the downscaled penalty terms $\eta_{min,e}:=c_{min}\eta_e$. We compute $c_{min}$ accurate to two decimal places using the bisection method. 

After we have computed $\eta$ and $c_{min}$, we consider the time step condition for the leap--frog scheme and use this to compute the time step size in the three ways given below: 
\begin{subequations}
\begin{align}
\Delta t(\eta_{min}) &:= \frac{2}{\sqrt{\lambda_{max}(M^{-1}A_{min})}}, \\
\Delta t(\eta) &:= \frac{2}{\sqrt{\lambda_{max}(M^{-1}A)}}, \\
\Delta t_{est}(\eta) &:= \frac{2}{\sqrt{\sup_{\omega\in\mathcal{W}_h} \lambda_{max}\big((M'_{\omega})^{-1}A'_{\omega}\big)}},
\end{align}
\label{eq:timeStepSizes}%
\end{subequations}
Here $A_{min}$ is the stiffness matrix that results from using the downscaled penalty terms $\eta_{min,e}$, and $M'_{\omega}$ and $A'_{\omega}$ are the submatrices corresponding to the weighted submesh $\omega$.  These time step sizes can be interpreted as follows: $\Delta t(\eta_{min})$ is the largest allowed time step size when using the minimum downscaled penalty terms $\eta_{min,e}$, $\Delta t(\eta)$ is the largest allowed time step size when using the penalty term estimates $\eta_e$, and $\Delta t_{est}(\eta)$ is the time step estimate when using the penalty term estimates $\eta_e$ and a weighted mesh decomposition. 

For our time step estimate $\Delta t_{est}(\eta)$ we will use the vertex-based mesh decomposition as given in (\ref{eq:nodalDomDec}). We will measure the sharpness of the penalty term by $\Delta t(\eta_{min}) / \Delta t(\eta)$, and we will measure the sharpness of our time step estimate by $\Delta t(\eta) / \Delta t_{est}(\eta)$.

\subsubsection{Regular meshes}
For the first tests we consider the acoustic wave equation as given in Example \ref{exm:acousticWave}, with $c=1$. We use meshes of the form described in Section \ref{sec:periodicMesh}, with element subdivisions as listed below. An illustration of some of the element subdivisions is given in Figure \ref{fig:elSubdiv1}.
\begin{itemize}
  \item \textit{1D}: mesh constructed from unit intervals.
  \item \textit{square}: 2D mesh constructed from unit squares.
  \item \textit{triangular}: 2D mesh constructed from unit squares, with each square subdivided into two triangles.
  \item \textit{cubic}: 3D mesh constructed from unit cubes.
  \item \textit{tetrahedral}: 3D mesh constructed from unit cubes, with each cube subdivided into six pyramids, and every pyramid subdivided into four tetrahedra.
\end{itemize}

The results of the parameter estimates, when using the penalty term $\eta^{*}$, are given in Table \ref{tab:regMesh1}. The results when using $\eta^{**}$ are given in Table \ref{tab:regMesh2}. From these tables we can already see that our second penalty term estimate $\eta^{**}$ is in general much sharper than the first estimate $\eta^{*}$. This is true especially for cubes, where $\eta^{*}$ causes a reduction in the largest allowed time step size of more than $2$. Also, for square and tetrahedral meshes the first penalty term estimate causes a reduction in the time step size of more than a factor $1.5$. On the other hand, when using $\eta^{**}$ the largest allowed time step size is never reduced more than a factor $1.2$, and for many of the regular meshes $\Delta t(\eta_{min}^{**})/ \Delta t(\eta^{**})$ is even below $1.01$. For the 1D, square, and cubic meshes, this penalty term and corresponding time step estimate even coincide with the analytic results derived in \cite{agut13}. In general, the time step estimate does not reduce the time step size by a factor more than $1.2$ when using $\eta^{**}$ and not more than $1.3$ when using $\eta^{*}$.

\begin{figure}[h]
\centering
\begin{subfigure}[b]{0.3\textwidth}
	\includegraphics[width=\textwidth]{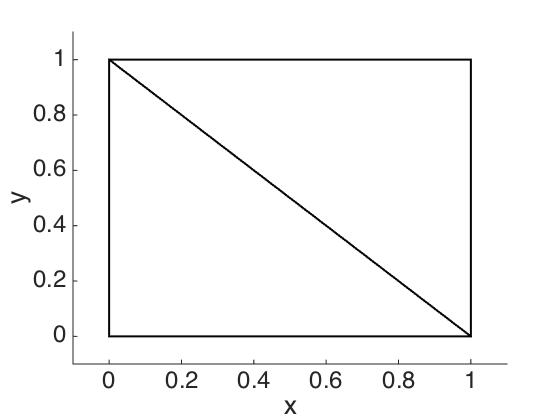}
        \caption{A square subdivided into two triangles.}
\end{subfigure} \;
\begin{subfigure}[b]{0.3\textwidth}
	\includegraphics[width=\textwidth]{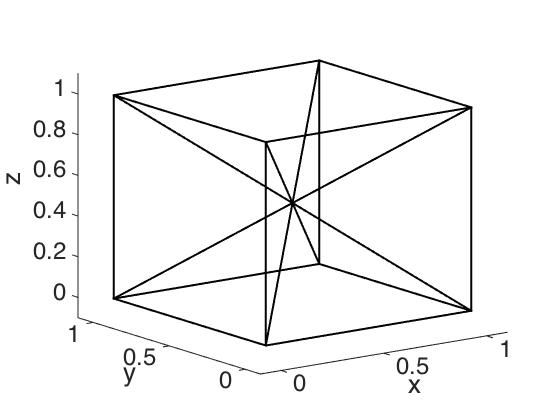}
        \caption{A cube subdivided into six pyramids.}
\end{subfigure} \;
\begin{subfigure}[b]{0.3\textwidth}
	\includegraphics[width=\textwidth]{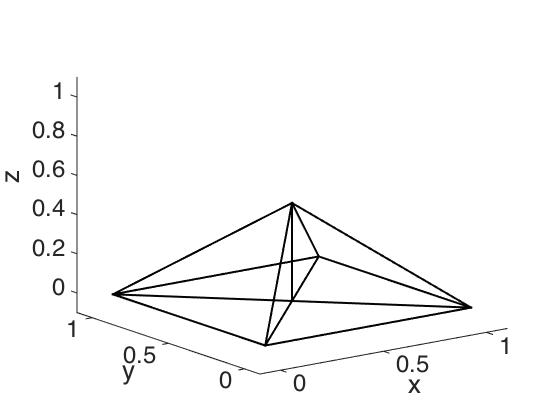}
        \caption{A pyramid subdivided into four tetrahedra.}
\end{subfigure}
\caption{}
\label{fig:elSubdiv1}
\end{figure}

\begin{table}[h]
\begin{center}
\begin{tabular}{c|c||c|c|c|c|c|c}
Mesh & p & $c_{min}^*$ & $\Delta t(\eta_{min}^*)$ &  $\Delta t(\eta^*)$ &  $\Delta t_{est}(\eta^*)$& $ \frac{\Delta t(\eta_{min}^*)}{ \Delta t(\eta^*) }$&  $\frac{\Delta t(\eta^*) }{ \Delta t_{est}(\eta^*) }$ \\ \hline\hline
1D 		& 1 & 1.00 & 0.5774 & 0.5774 & 0.5774 & 1.00 & 1.00 \\
		& 2 & 1.00 & 0.2582 & 0.2582 & 0.2582 & 1.00 & 1.00 \\
		& 3 & 1.00 & 0.1533 & 0.1533 & 0.1533 & 1.00 & 1.00 \\ \hline
square 	& 1 & 0.25 & 0.4082 & 0.2357 & 0.2019 & 1.73 & 1.17 \\
		& 2 & 0.33 & 0.1826 & 0.1170 & 0.0956 & 1.56 & 1.22 \\
		& 3 & 0.38 & 0.1084 & 0.0694 & 0.0554 & 1.56 & 1.25 \\ \hline
triangular	& 1 & 0.67 & 0.2579 & 0.2273 & 0.1948 & 1.13 & 1.17 \\
		& 2 & 0.69 & 0.1406 & 0.1250 & 0.1048 & 1.12 & 1.19 \\
		& 3 & 0.70 & 0.0906 & 0.0739 & 0.0621 & 1.23 & 1.19 \\ \hline
cubic 	& 1 & 0.14 & 0.3333 & 0.1361 & 0.1172 & 2.45 & 1.16 \\
		& 2 & 0.20 & 0.1491 & 0.0678 & 0.0554 & 2.20 & 1.22 \\
		& 3 & 0.23 & 0.0885 & 0.0405 & 0.0322 & 2.19 & 1.26 \\ \hline
tetrahedral&1 & 0.38 & 0.1035 & 0.0635 & 0.0560 & 1.63 & 1.13 \\
		& 2 & 0.44 & 0.0598 & 0.0384 & 0.0336 & 1.56 & 1.14 \\
		& 3 & 0.48 & 0.0360 & 0.0243 & 0.0212 & 1.48 & 1.15		
\end{tabular}
\end{center}
\caption{Parameter estimates on regular meshes, using penalty term $\eta^*$.}
\label{tab:regMesh1}
\end{table}

\begin{table}[h]
\begin{center}
\begin{tabular}{c|c||c|c|c|c|c|c}
Mesh & p & $c_{min}^{**}$ & $\Delta t(\eta_{min}^{**})$ &  $\Delta t(\eta^{**})$ &  $\Delta t_{est}(\eta^{**})$& $ \frac{\Delta t(\eta_{min}^{**})}{ \Delta t(\eta^{**}) }$&  $\frac{\Delta t(\eta^{**}) }{ \Delta t_{est}(\eta^{**}) }$ \\ \hline\hline
1D 		& 1 & 1.00 & 0.5774 & 0.5774 & 0.5774 & 1.00 & 1.00 \\
		& 2 & 1.00 & 0.2582 & 0.2582 & 0.2582 & 1.00 & 1.00 \\
		& 3 & 1.00 & 0.1533 & 0.1533 & 0.1533 & 1.00 & 1.00 \\ \hline
square 	& 1 & 1.00 & 0.4082 & 0.4082 & 0.4082 & 1.00 & 1.00 \\
		& 2 & 1.00 & 0.1826 & 0.1826 & 0.1826 & 1.00 & 1.00 \\
		& 3 & 1.00 & 0.1084 & 0.1084 & 0.1084 & 1.00 & 1.00 \\ \hline
triangular	& 1 & 1.00 & 0.2582 & 0.2582 & 0.2427 & 1.00 & 1.06 \\
		& 2 & 0.96 & 0.1406 & 0.1399 & 0.1275 & 1.01 & 1.10 \\
		& 3 & 0.96 & 0.0906 & 0.0896 & 0.0755 & 1.01 & 1.19 \\ \hline
cubic 	& 1 & 1.00 & 0.3333 & 0.3333 & 0.3333 & 1.00 & 1.00 \\
		& 2 & 1.00 & 0.1491 & 0.1491 & 0.1491 & 1.00 & 1.00 \\
		& 3 & 1.00 & 0.0885 & 0.0885 & 0.0885 & 1.00 & 1.00 \\ \hline
tetrahedral& 1 & 0.75 & 0.1040 & 0.0918 & 0.0803 & 1.13 & 1.14 \\
		& 2 & 0.74 & 0.0599 & 0.0510 & 0.0455 & 1.17 & 1.15 \\
		& 3 & 0.81 & 0.0359 & 0.0320 & 0.0279 & 1.12 & 1.15	
\end{tabular}
\end{center}
\caption{Parameter estimates on regular meshes, using penalty term $\eta^{**}$.}
\label{tab:regMesh2}
\end{table}

\begin{table}[h]
\begin{center}
\begin{tabular}{c|c||c|c|c|c|c}
Mesh & p & $\Delta t(\eta)$ &$\Delta t(\eta^*)$ &$\Delta t(\eta^{**})$ & $\frac{ \Delta t(\eta^*) }{ \Delta t(\eta) }$ & $\frac{ \Delta t(\eta^{**}) }{ \Delta t(\eta) }$ \\ \hline\hline
 triangular		& 1 & 0.2280 & 0.2273 & 0.2582 & 1.00 & 1.13 \\
			& 2 & 0.1002 & 0.1250 & 0.1399 & 1.25 & 1.40 \\
			& 3 & 0.0567 & 0.0739 & 0.0896 & 1.30 & 1.58 \\ \hline
tetrahedral	& 1 & 0.0689 & 0.0635 & 0.0918 & 0.92 & 1.33 \\
			& 2 & 0.0327 & 0.0384 & 0.0510 & 1.17 & 1.56 \\
			& 3 & 0.0196 & 0.0243 & 0.0320 & 1.24 & 1.63
\end{tabular}
\end{center}
\caption{Parameter estimates on a regular tetrahedral mesh, using the penalty term, here denoted by $\eta$, derived in \cite{mulder14}, and using penalty terms $\eta^*$ and $\eta^{**}$.}
\label{tab:regMesh3}
\end{table}

For the case of tetrahedral meshes we can compare our penalty term estimates with the estimate derived in \cite{mulder14}. The penalty term derived there is equivalent to $\eta_e = p(p+2)$, with the penalty scaling function given by $\nu_h|_{\partial e\cap f}=1/(\epsilon_f\min_{e\in{\mathcal{T}_f}}d_{i,e} )$, where $d_{i,e}$ is the diameter of the inscribed sphere of $e$ and $\epsilon_f\in\{1/2,1\}$ is defined as in (\ref{eq:aIPdef2}). Their analysis can be readily extended to triangles by replacing the trace inverse inequality for tetrahedra by the trace inverse inequality for triangles, given in Theorem 3 of \cite{warburton03}, which is equivalent to setting $\eta_e=p(p+1)$. The results of these estimates are given in Table \ref{tab:regMesh3}, from which we can see that this penalty term estimate has a similar sharpness as $\eta^*$, but is significantly less sharp than $\eta^{**}$, having a time step size more than $1.5$ times smaller than when using $\eta^{**}$ for $p=2,3$ on tetrahedra and $p=3$ on triangles.

Since $\eta^{**}$ is significantly sharper than $\eta^{*}$, we will only use $\eta^{**}$ throughout the following numerical tests.

\subsubsection{Meshes with deformed elements}
In this subsection we consider the acoustic wave equation with $c=1$ again,  but now using deformed elements. For the penalty term we will only use $\eta^{**}$. An overview of the different meshes is listed below, and an illustration of the element subdivisions is given in Figures \ref{fig:elSubdiv2a} and \ref{fig:elSubdiv2b}.

\begin{itemize}
\item \textit{rectangular}$[x]$: 2D mesh constructed from unit squares, with each square subdivided into $2\times 2$ rectangles adjacent to a central node at $(x,x)$.
\item \textit{quadrilateral}$[x]$: 2D mesh constructed from unit squares, with each square subdivided into $2\times2$ smaller uniform squares, after which the central node at (0.5,0.5) is moved to $(x,x)$.
\item \textit{triangular}$[x]$: 2D mesh constructed from unit squares, with each square subdivided into four triangles adjacent to the central node at $(x,x)$.
\item \textit{cuboid}$[x]$: 3D mesh constructed from unit cubes, with each cube subdivided into $2\times2\times2$ cuboids adjacent to the central node at $(x,x,x)$. 
\item \textit{hexahedral}$[x]$: 3D mesh constructed from unit cubes, with each cube subdivided into $2\times2\times2$ smaller uniform cubes, after which the central node at $(0.5,0.5,0.5)$ is moved to $(x,x,x)$. 
\item \textit{tetrahedral}$[x]$: 3D mesh constructed from unit cubes, with each cube subdivided into six pyramids adjacent to the central node at $(x,x,x)$, and with each pyramid subdivided into four tetrahedra.
\end{itemize}

\begin{figure}[h]
\centering
\begin{subfigure}[b]{0.3\textwidth}
	\includegraphics[width=\textwidth]{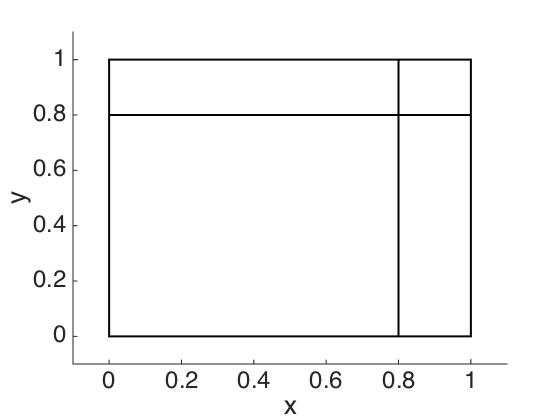}
        \caption{A square divided into four rectangles with a central node at (0.8,0.8).}
\end{subfigure} \;
\begin{subfigure}[b]{0.3\textwidth}
	\includegraphics[width=\textwidth]{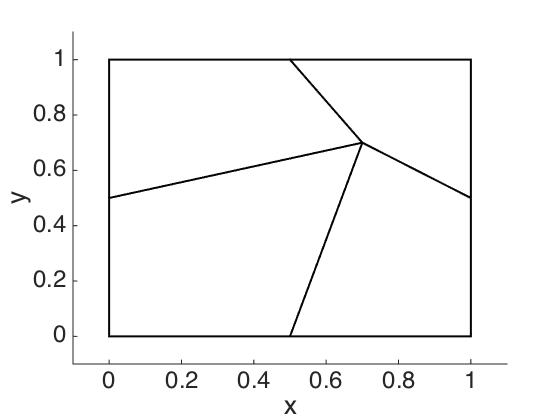}
        \caption{A square divided into four quadrilaterals with a central node at (0.7,0.7).}
\end{subfigure} \;
\begin{subfigure}[b]{0.3\textwidth}
	\includegraphics[width=\textwidth]{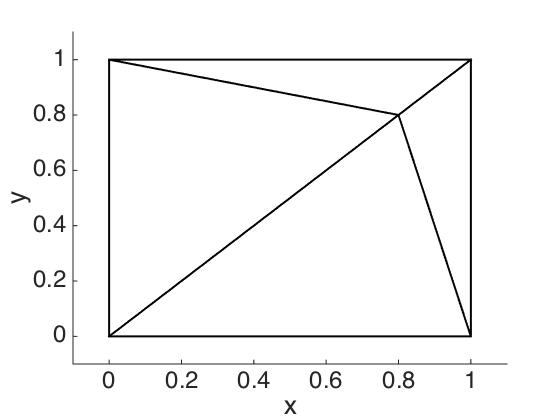}
        \caption{A square divided into four triangles with a central node at (0.7,0.7).}
\end{subfigure}
\caption{}
\label{fig:elSubdiv2a}
\end{figure}

\begin{figure}[h]
\centering
\begin{subfigure}[t]{0.3\textwidth}
	\includegraphics[width=\textwidth]{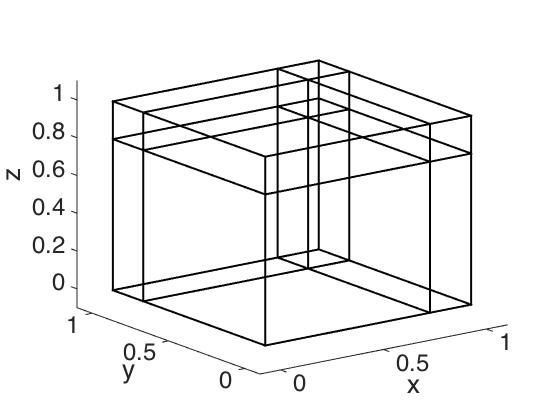}
        \caption{A cube divided into eight cuboids with a central node at (0.8,0.8,0.8).}
\end{subfigure} \;
\begin{subfigure}[t]{0.3\textwidth}
	\includegraphics[width=\textwidth]{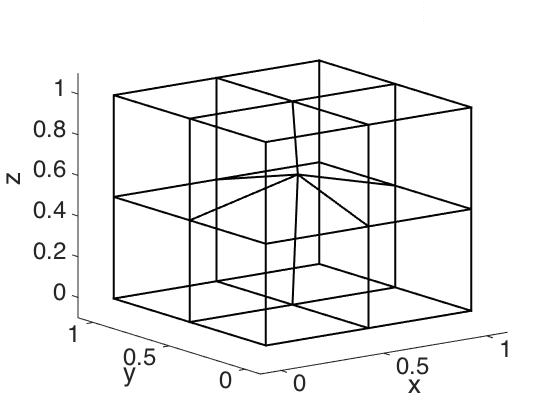}
        \caption{A cube divided into eight hexahedra with a central node at (0.6,0.6,0.6).}
\end{subfigure} \;
\begin{subfigure}[t]{0.3\textwidth}
	\includegraphics[width=\textwidth]{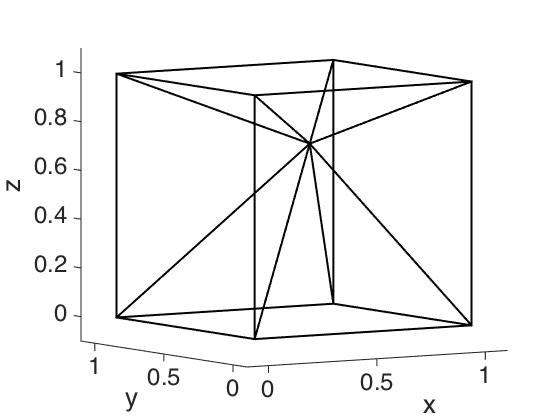}
        \caption{A cube divided into six pyramids with a central node at (0.7,0.7,0.7).}
\end{subfigure}
\caption{}
\label{fig:elSubdiv2b}
\end{figure}

\begin{table}[h]
\begin{center}
\begin{tabular}{c|c|c||c|c}
Mesh & x & p & $ \frac{\Delta t(\eta_{min}^{**})}{ \Delta t(\eta^{**}) }$&  $\frac{\Delta t(\eta^{**}) }{ \Delta t_{est}(\eta^{**}) }$ \\ \hline\hline
triangular$[x]$ 		& $\{0.5,0.7,0.9\}$	& 1 & [1.04, 1.06] & [1.14, 1.20] \\
				&			& 2 & [1.05, 1.09] & [1.18, 1.20] \\
				&			& 3 & [1.05, 1.09] & [1.19, 1.21] \\ \hline
rectangular$[x]$	&$\{0.5,0.7,0.9\}$	& 1 & [1.00, 1.00] & [1.00, 1.11] \\
				&			& 2 & [1.00, 1.00] & [1.00, 1.28] \\
				&			& 3 & [1.00, 1.00] & [1.00, 1.37] \\ \hline
quadrilateral$[x]$	& $\{0.5,0.6,0.7\}$	& 1 & [1.00, 1.05] & [1.00, 1.20] \\
				&			& 2 & [1.00, 1.04] & [1.00, 1.26] \\
				&			& 3 & [1.00, 1.05] & [1.00, 1.31]
\end{tabular}
\end{center}
\caption{Parameter estimates on deformed 2D meshes, using penalty term $\eta^{**}$. Intervals denote the range in which the time step ratios are found.}
\label{tab:defMesh2D}
\end{table}

\begin{table}[h]
\begin{center}
\begin{tabular}{c|c|c||c|c}
Mesh & x & p & $ \frac{\Delta t(\eta_{min}^{**})}{ \Delta t(\eta^{**}) }$&  $\frac{\Delta t(\eta^{**}) }{ \Delta t_{est}(\eta^{**}) }$ \\ \hline\hline
tetrahedral$[x]$		& $\{0.5,0.7,0.9\}$	& 1 & [1.06, 1.13] & [1.12, 1.14] \\
				& 			& 2 & [1.08, 1.17] & [1.14, 1.17] \\
				& 			& 3 & [1.07, 1.12] & [1.14, 1.15] \\ \hline
cuboid$[x]$		& $\{0.5,0.7,0.9\}$	& 1 & [1.00, 1.00] & [1.00, 1.11] \\
				&			& 2 & [1.00, 1.00] & [1.00, 1.28] \\
				&			& 3 & [1.00, 1.00] & [1.00, 1.37] \\ \hline
hexahedral$[x]$	& $\{0.5,0.6,0.65\}$	& 1 & [1.00, 1.09] & [1.00, 1.17] \\
				&			& 2 & [1.00, 1.07] & [1.00, 1.25] \\
				& 			& 3 & [1.00, 1.10] & [1.00, 1.28]
\end{tabular}
\end{center}
\caption{Parameter estimates on deformed 3D meshes, using penalty term $\eta^{**}$. Intervals denote the range in which the time step ratios are found.}
\label{tab:defMesh3D}
\end{table}

The results of the parameter estimates are given in Tables \ref{tab:defMesh2D} and \ref{tab:defMesh3D}. In all cases, the penalty term estimate causes a reduction in the time step size of no more than a factor $1.1$, with respect to the time step size using the minimal downscaled penalty term $\eta_{min}^{**}$. The time step estimates for triangular meshes causes a reduction of no more than $1.25$ and for the tetrahedral meshes a reduction of no more than $1.2$ with respect to the largest allowed time step size. For quadrilateral and hexahedral meshes the time step estimate tends to get less sharp for higher order polynomial basis functions and more strongly deformed meshes, but in all of our tests $\Delta t(\eta^{**})/ \Delta t_{est}(\eta^{**})$ remains below $1.4$.

Since it is hard to compute the element and face integrals for general quadrilateral and hexahedral meshes exactly, we approximate them using the Gauss--Legendre quadrature rule with $p+1$ points in every direction, where $p$ is the polynomial order of the basis functions. We do not consider meshes of the type \textit{quadrilateral}$[x]$ with $x\geq 0.75$ and \textit{hexahedral}$[x]$ with $x\geq 2/3$ since in those cases one of the elements no longer has a well-defined mapping.

\subsubsection{Meshes with piecewise linear parameters}
We now consider the acoustic wave equation with piecewise linear parameter fields $\rho$ and $c$ instead of constant parameters. An overview of the different meshes is listed below.

\begin{itemize}
\item \textit{squarePL}$[\rho_0,c_0]$: 2D mesh constructed from unit squares, with each square subdivided into $2\times 2$ smaller squares and with piecewise linear parameters $\rho,c$ such that $\rho=c=1$ at $y=0$ and $y=1$ and $\rho=\rho_0, c=c_0$ at $y=0.5$.
\item \textit{triPL}$[\rho_0,c_0]$: 2D mesh constructed from unit squares, with each square subdivided into four uniform triangles and with piecewise linear parameters $\rho,c$ such that $\rho=c=1$ at the boundary and $\rho=\rho_0, c=c_0$ at the center of each square.
\item \textit{cubicPL}$[\rho_0,c_0]$: 3D mesh constructed from unit cubes, with each cube subdivided into $2\times2\times2$ smaller cubes and with piecewise linear parameters $\rho,c$ such that $\rho=c=1$ at $z=0$ and $z=1$ and $\rho=\rho_0, c=c_0$ at $z=0.5$. 
\item \textit{tetrahedralPL}$[\rho_0,c_0]$: 3D mesh constructed from unit cubes, with each cube subdivided into $24$ uniform tetrahedra and with piecewise linear parameters $\rho,c$ such that $\rho=c=1$ at the boundary and $\rho=\rho_0, c=c_0$ at the center of each cube.
\end{itemize}

\begin{table}[h]
\begin{center}
\begin{tabular}{c|c|c||c|c}
Mesh & $(\rho_0,c_0)$ & p & $ \frac{\Delta t(\eta_{min}^{**})}{ \Delta t(\eta^{**}) }$&  $\frac{\Delta t(\eta^{**}) }{ \Delta t_{est}(\eta^{**}) }$ \\ \hline\hline
triPL$[\rho_0,c_0]$ & $\{(1,1), (5,5), (10,1), ...$ 	& 1 & [1.01,1.09] & [1.04,1.17] \\
& $\phantom{\{}(1,10), (10,10)\}$		& 2 & [1.01,1.11] & [1.06,1.19] \\
&								& 3 & [1,01,1.09] & [1.08,1.20] \\ \hline
squarePL$[\rho_0,c_0]$ & $\{(1,1), (5,5), (10,1),...$	& 1 & [1.00,1.00] & [1.00,1.00] \\
&$\phantom{\{}(1,10), (10,10)\}$		& 2 & [1.00,1.00] & [1.00,1.00] \\
&								& 3 & [1.00,1.00] & [1.00,1.04] \\ \hline
tetraPL$[\rho_0,c_0]$ & $\{(1,1), (5,5), (10,1),...$ 	& 1 & [1.12,1.19] & [1.13,1.14] \\
&$\phantom{\{}(1,10), (10,10)\}$		& 2 & [1.12,1.19] & [1.13,1.15] \\
&								& 3 & [1.10,1.12] & [1.14,1.15] \\ \hline
cubicPL$[\rho_0,c_0]$ & $\{(1,1), (5,5), (10,1), ...$	& 1 & [1.00,1.00] & [1.00,1.00] \\
&$\phantom{\{}(1,10), (10,10)\}$		& 2 & [1.00,1.00] & [1.00,1.00] \\
&								& 3 & [1.00,1.00] & [1.00,1.03] 
\end{tabular}
\end{center}
\caption{Parameter estimates on 2D and 3D meshes for the acoustic wave equation with piecewise linear parameters $\rho$ and $c$, using penalty term $\eta^{**}$. Intervals denote the range in which the time step ratios are found.}
\label{tab:PLMesh}
\end{table}

The results of the parameter estimates are given in Table \ref{tab:PLMesh}. The penalty term estimate cause a reduction in the time step size of no more than a factor $1.2$, with respect to the time step size using the minimal downscaled penalty term $\eta_{min}^{**}$ for triangular and tetrahedral elements and no reduction at all for square and cubic elements. The time step estimates for the triangular and tetrahedral meshes causes a reduction in the time step size of no more than $1.2$ with respect to the largest allowed time step size, while for the square and cubic meshes they often cause no reduction at all or a reduction not larger than $1.05$.

\subsubsection{Meshes for electromagnetic and elastic wave problems}
In this subsection we consider the electromagnetic wave equations as given in Example \ref{exm:Maxwell} and the 3D isotropic elastic wave equations as given in Example \ref{exm:anisotropicWave}. We use the electromagnetic wave equations to test heterogeneous media with sharp contrasts in material parameters, while we use the elastic wave equations to test media with large contrasts in primary and secondary wave velocities. An overview of the different meshes is listed below.

\begin{itemize}
\item \textit{tetraEM}$[\mu_1, \mu_2]$: 3D mesh constructed from unit cubes, with each cube subdivided into $24$ smaller uniform tetrahedra, where all tetrahedra below the surface $x=z$ have a relative magnetic permeability of $\mu_1$ and all tetrahedra above this surface have a permeability of $\mu_2$. For all tetrahedra the relative electric permittivity $\epsilon$ equals $1$.
\item \textit{cubicEM}$[\mu_1, \mu_2]$: 3D mesh constructed from unit cubes, with each cube subdivided into $2\times2\times2$ smaller uniform cubes, where the bottom four cubes have a relative magnetic permeability of $\mu_1$ and the top four cubes a permeability of $\mu_2$. For all cubes the relative electric permittivity $\epsilon$ equals $1$.
\item \textit{tetraISO}$[\lambda, \mu]$: 3D mesh constructed from unit cubes, with each cube subdivided into $24$ smaller uniform tetrahedra, and with all tetrahedra having Lam\'e parameters $\lambda$ and $\mu$ and mass density $\rho=1$.
\item \textit{cubicISO}$[\lambda, \mu]$: 3D mesh constructed from unit cubes, with each cube having Lam\'e parameters $\lambda$ and $\mu$ and mass density $\rho=1$.
\end{itemize}

\begin{table}[h]
\begin{center}
\begin{tabular}{c|c|c|c||c|c}
Mesh & $\mu_1^{-1}$ & $\mu_2^{-1}$ & p & $ \frac{\Delta t(\eta_{min}^{**})}{ \Delta t(\eta^{**}) }$&  $\frac{\Delta t(\eta^{**}) }{ \Delta t_{est}(\eta^{**}) }$ \\ \hline\hline
tetraEM$[\mu_1,\mu_2]$ & 1 & $\{1, 10, 100\}$ 	& 1 & [1.04, 1.05] & [1.14, 1.15] \\
&&								& 2 & [1.04, 1.07] & [1.15, 1.15] \\
&&								& 3 & [1.00, 1.04] & [1.15, 1.16] \\ \hline
cubicEM$[\mu_1,\mu_2]$ & 1 & $\{1, 10, 100\}$	& 1 & [1.00, 1.07] & [1.29, 1.29] \\
&&								& 2 & [1.00, 1.00] & [1.31, 1.35] \\
&&								& 3 & [1.00, 1.01] & [1.33, 1.35] 
\end{tabular}
\end{center}
\caption{Parameter estimates on a 3D electromagnetic mesh, using penalty term $\eta^{**}$. Intervals denote the range in which the time step ratios are found.}
\label{tab:EMMesh3D}
\end{table}

\begin{table}[h]
\begin{center}
\begin{tabular}{c|c|c||c|c}
Mesh & $(\lambda,\mu)$ & p & $ \frac{\Delta t(\eta_{min}^{**})}{ \Delta t(\eta^{**}) }$&  $\frac{\Delta t(\eta^{**}) }{ \Delta t_{est}(\eta^{**}) }$ \\ \hline\hline
tetraISO$[\lambda,\mu]$ & $\{(1,0),(0,1),(10,1),(100,1)\}$	& 1 & [1.00, 1.11] & [1.14, 1.15] \\
&									& 2 & [1.00, 1.14] & [1.14, 1.15] \\
&									& 3 & [1.00, 1.13] & [1.13, 1.15] \\ \hline
cubicSO$[\lambda,\mu]$ & $\{(1,0),(0,1),(10,1),(100,1)\}$	& 1 & [1.05, 1.20] & [1.24, 1.28] \\
&									& 2 & [1.00, 1.10] & [1.23, 1.31] \\
&									& 3 & [1.00, 1.07] & [1.24, 1.33] 
\end{tabular}
\end{center}
\caption{Parameter estimates on a 3D isotropic elastic tetrahedral mesh, using penalty term $\eta^{**}$. Intervals denote the range in which the time step ratios are found.}
\label{tab:IsoMesh3D}
\end{table}

The results of the parameter estimates are given in Tables \ref{tab:EMMesh3D} and \ref{tab:IsoMesh3D}. For the heterogeneous electromagnetic problems, the penalty term estimate causes a reduction in the time step size of no more than a factor $1.1$, with respect to the time step size using the minimal downscaled penalty term $\eta_{min}^{**}$. For the isotropic elastic meshes this reduction is no more than a factor $1.2$. The time step estimates for all the tetrahedral meshes causes a reduction in the time step size of no more than $1.25$ with respect to the largest allowed time step size, while for the cubic meshes this reduction remains below a factor $1.4$.

\section{Conclusion}
\label{sec:conclusion}
We have derived sharp and sufficient conditions for the penalty parameter in Theorem \ref{thm:condPen1} and Theorem \ref{thm:condPen2} to guarantee stability of the SIPDG method for linear wave problems. In addition, we derived sufficient upper bounds for the spectral radius and the time step size in Theorem \ref{thm:spectralRadiusDG} and Section \ref{sec:weightedDomDec}, by introducing a weighted mesh decomposition. These conditions hold for generic meshes, including unstructured nonconforming heterogeneous meshes of mixed element types, with different types of boundary conditions. Moreover, the estimates hold for general linear hyperbolic partial differential equations, including the acoustic wave equation, the (an)isotropic elastic wave equations, and Maxwell's equations. Both the penalty term and time step size can be efficiently computed in parallel using a power iteration method on each element and submesh, respectively.


To test the sharpness of our estimates we have considered several semiuniform meshes made of $N^d$ uniform cubes or squares, which are then uniformly subdivided into smaller meshes, including meshes with deformed elements and heterogeneous media. From the results we can see that the penalty term estimate given in Theorem \ref{thm:condPen2} is significantly sharper than the estimate in Theorem \ref{thm:condPen1} for 2D and 3D meshes. Especially for cubic elements, we see that downscaling the penalty estimate constructed using Theorem \ref{thm:condPen1} allows a time step size more than twice as large, while the penalty term in Theorem \ref{thm:condPen2} does not allow any downscaling at all. Furthermore, the penalty term estimate from Theorem \ref{thm:condPen2} allows a time step size more than a factor $1.5$ times larger than when using the penalty term estimate derived in \cite{mulder14}, when using square or cubic basis functions on tetrahedra or cubic basis functions on triangles. For regular square and cubic meshes, this same penalty term estimate, together with the time step estimate from Theorem \ref{thm:spectralRadiusDG} using the weighted mesh decomposition of (\ref{eq:nodalDomDec}), conforms with the largest allowed time step sizes analytically derived in \cite{agut13}.

In general we see that downscaling the penalty term from Theorem \ref{thm:condPen2} does not increase the maximum allowed time step size by more than a factor $1.2$, even when the material parameters are nonconstant within the elements, while our time step estimate, based on a vertex-based weighted mesh decomposition, does not become smaller than a factor $1.2$ compared to the maximum allowed time step size for triangular and tetrahedral meshes, and not smaller than a factor $1.4$ for quadrilateral and hexahedral meshes.

\section*{Acknowledgments}
We would like to thank Shell Global Solutions International B.V. for supporting this project, and in particular Wim Mulder for giving us useful advice and sharing his practical experiences related to the topic.

\bibliographystyle{abbrv}
\bibliography{ParameterEstimates}

\appendix
\section{Linear Algebra}
\begin{lem} 
Let $C\in\mathbb{R}^{d\times m\times m\times d}_{sym}$ be a symmetric tensor such that
\begin{align*}
\tsigma^t:C:\tsigma &\geq 0, && \tsigma\in\mathbb{R}^{d\times m}.
\end{align*}
Then there exists a symmetric tensor field $C^{1/2}\in\mathbb{R}_{sym}^{d\times m\times m\times d}$ such that $C^{1/2}:C^{1/2} = C$.
\label{lem:Csqrt}
\end{lem}
\begin{proof}
The result follows from the fact that $\tsigma\rightarrow C:\tsigma$ is a linear, self-adjoint, and positive semidefinite operator on a finite-dimensional subspace $\mathbb{R}^{d\times m}$.
\end{proof}

\begin{lem}
Let $\vn\in\mathbb{R}^d$ be a vector and $C\in\mathbb{R}^{d\times m\times m\times d}_{sym}$ be a symmetric tensor such that
\begin{align}
\tsigma^t:C:\tsigma &\geq 0, && \tsigma\in\mathbb{R}^{d\times m}.
\label{eq:int1CnInv}
\end{align}
Also, define the second order tensor $\ten{c}_{\vn} := \vn\cdot C\cdot\vn$ and the following function space:
\begin{align*}
U &:= \big\{ \vu\in\mathbb{R}^m \;\big|\; \vu = \vn\cdot C:\tsigma, \text{ for some }\tsigma\in\mathbb{R}^{d\times m} \big\}.
\end{align*}
There exists a pseudoinverse $\ten{c}_{\vn}^{-1}\in\mathbb{R}^{m\times m}$ such that $\ten{c}_{\vn}^{-1}\cdot\ten{c}_{\vn}\cdot\vu = \ten{c}_{\vn}\cdot\ten{c}_{\vn}^{-1}\cdot\vu = \vu$, for all $\vu\in U$. Moreover, both $\ten{c}_{\vn}$ and $\ten{c}_{\vn}^{-1}$ are symmetric and positive semidefinite.
\label{lem:CnInv}
\end{lem}

\begin{proof}
The fact that $\ten{c}_{\vn}$ is symmetric follows from its definition and the fact that $C$ is symmetric. The fact that this tensor is also positive semidefinite follows from (\ref{eq:int1CnInv}) and by writing
\begin{align*}
\vu\cdot\ten{c}_{\vn}\cdot\vu &= \vu\vn : C : \vn\vu \geq 0, && \vu\in \mathbb{R}^m.
\end{align*}
To prove that the pseudoinverse with respect to $U$ exists and is symmetric positive semidefinite, it is sufficient to show that
\begin{align}
\vu\cdot\ten{c}_{\vn}\cdot\vu &>0, &&\vu\in  U,\vu\neq\vzero.
\label{eq:int2CnInv}
\end{align}
Since $\ten{c}_{\vn}$ is positive semidefinite we then only need to show that 
\begin{align}
\vu\in U \;\wedge\; \vu\cdot\ten{c}_{\vn}\cdot\vu &=0   \quad\Longrightarrow\quad \vu = \vzero.
\label{eq:int3CnInv}
\end{align}
Now take an arbitrary $\vu\in U$, and suppose that $\vu\cdot\ten{c}_{\vn}\cdot\vu =0$. Using Lemma \ref{lem:Csqrt} we can write
\begin{align*}
0=\vu\cdot\ten{c}_{\vn}\cdot\vu &= \| C^{1/2}:\vn\vu \|^2.
\end{align*}
From this it follows that $C:\vn\vu = \tzero$. Because of the definition of $U$ we can write $\vu = \vn\cdot C:\tsigma$ for some $\tsigma\in\mathbb{R}^{d\times m}$, and therefore $C:\vn(\vn\cdot C:\tsigma)=\tzero$. Applying the double dot product with $\tsigma^t$ on the left side gives
\begin{align*}
0 &= \tsigma^t:C:\vn(\vn\cdot C:\tsigma) 
=  (\tsigma^t:C\cdot\vn)\cdot(\vn\cdot C:\tsigma) 
= \| \vu \|^2.
\end{align*}
Therefore, $\vu=\vzero$, which proves (\ref{eq:int3CnInv}).
\end{proof}

\begin{lem}
Let $\{a_i\}_{i=1}^n$ be a set of nonnegative scalars, and let $\{ b_i \}_{i=1}^n$ be a set of positive scalars. Then
\begin{align*}
\frac{ \sum_{i=1}^n a_i}{ \sum_{i=1}^n b_i} \leq \max_{i=1,\dots,n} \frac{a_i}{b_i}.
\end{align*}
\label{lem:ineqFrac}
\end{lem}
\begin{proof}
\begin{align*}
\frac{ \sum_{i=1}^n a_i}{ \sum_{i=1}^n b_i}  = \sum_{i=1}^n \frac{a_i}{b_i} \frac{b_i}{\sum_{j=1}^n b_j} &\leq  \max_{k=1,\dots,n} \frac{a_k}{b_k} \sum_{i=1}^n \frac{b_i}{\sum_{j=1}^n b_j} = \max_{k=1,\dots,n} \frac{a_k}{b_k} .
\end{align*}
\end{proof}

\begin{lem}
Let $A\in\mathbb{R}^{n\times n}_{sym}$ be a symmetric matrix and $M\in\mathbb{R}^{n\times n}_{sym}$ a positive definite matrix. Then there exists a diagonalization $M^{-1}A=VDV^{-1}$, such that $D$ is a real diagonal matrix and $V$ satisfies $V^tMV = I$, where $I$ is the identity matrix.
\label{lem:propMinvA}
\end{lem}

\begin{proof}
Since $M$ is symmetric positive definite, there exists a symmetric positive definite matrix $M^{1/2}$ such that $M^{1/2}M^{1/2}=M$. Define $\tilde{A}:=M^{-1/2}AM^{-1/2}$. Since $A$ is symmetric, $\tilde{A}$ is symmetric as well and can be diagonalised as $\tilde A = \tilde{V}D\tilde{V}^{-1}$ with $D$ a diagonal real matrix and $\tilde V$ satisfying $\tilde{V}^t\tilde{V}=I$. Now define $V:=M^{-1/2}\tilde{V}$. Then $M^{-1}A=V^{-1}DV$ and $V^tMV=I$.
\end{proof}

\begin{lem}
Let $A\in\mathbb{R}^{n\times n}_{sym}$ be a symmetric matrix and $M\in\mathbb{R}^{n\times n}_{sym}$ a positive definite matrix. Then
\begin{align*}
\sup_{\vvct{u}\in\mathbb{R}^n, \vvct{u}\neq \vvct{0} }\frac{ |\vvct{u}^t A \vvct{u}| }{ \vvct{u}^tM\vvct{u} } &= \lambda_{max}\big( M^{-1}A \big), 
\end{align*}
where $\lambda_{max}\big( M^{-1}A \big)$ denotes the largest eigenvalue of $M^{-1}A$ in magnitude.
\label{lem:largestEigValue}
\end{lem}

\begin{proof}
From Lemma \ref{lem:propMinvA} it follows that there exist a diagonal matrix $D$ and a nonsingular matrix $V$ such that $MVDV^{-1} = A$ and $V^tMV = I$. Now consider an arbitrary $\vvu\in\mathbb{R}^n, \vvu\neq\vvct{0}$. We set $\vvct{w}:=V^{-1}\vvct{u}$ to obtain
\begin{align*}
\frac{ |\vvct{u}^t A \vvct{u}| }{ \vvct{u}^tM\vvct{u} }&= \frac{ |\vvct{w}^t V^tMVDV^{-1}V \vvct{w}| }{ \vvct{w}^tV^tMV\vvct{w} } 
= \frac{ |\vvct{w}^t D \vvct{w}| }{ \vvct{w}^t\vvct{w} }.
\end{align*}
The lemma follows from this equality and the following relation:
\begin{align*}
 \sup_{\vvct{w}\in\mathbb{R}^n, \vvct{w}\neq \vvct{0} }\frac{ |\vvct{w}^t D \vvct{w}| }{ \vvct{w}^t\vvct{w} }&=  \max_{i=1,\dots,N} |D_{ii}| = \lambda_{max}\big( M^{-1}A \big).
\end{align*}
\end{proof}

\end{document}